\documentclass[11pt]{amsart} 

\usepackage{amsmath,amssymb,amsthm,amsfonts,verbatim}
\usepackage{enumerate}
\usepackage{color}
\usepackage{graphicx}

\usepackage{tikz-cd}
\usepackage[utf8]{inputenc}

\theoremstyle{plain}
\newtheorem{theorem}{Theorem}[section]

\newtheorem{proposition}[theorem]{Proposition}

\newtheorem{lemma}[theorem]{Lemma}

\newtheorem{remark}[theorem]{Remark}

\newtheorem{corollary}[theorem]{Corollary}
\theoremstyle{definition}
\newtheorem{definition}[theorem]{Definition}

\newcommand{\nc}{\newcommand}
\nc{\dmo}{\DeclareMathOperator}

\nc{\Q}{\mathbb{Q}}
\nc{\R}{\mathbb{R}}
\nc{\RP}{\mathbb{RP}^1}
\nc{\Z}{\mathbb{Z}}
\nc{\ZZ}{\mathbb{Z}}
\nc{\C}{\mathbb{C}}
\nc{\cS}{\mathcal{S}}
\nc{\iso}{\cong}
\dmo{\Mod}{Mod}
\dmo{\Ig}{\mathcal{I}_g}
\dmo{\Span}{span}
\dmo{\Diff}{Diff}
\dmo{\Homeo}{Homeo}
\dmo{\dist}{dist}
\dmo\BDiff{BDiff}
\dmo\SO{SO}
\dmo\slide{sl}
\dmo\im{im}
\dmo\id{id}
\dmo\Fix{Fix}
\dmo\Stab{Stab}
\dmo\Mcg{Mcg}
\dmo{\Hg}{\mathcal{H}_g}
\dmo{\Tg}{\mathcal{T}_g}

\renewcommand{\epsilon}{\varepsilon}
\nc{\coloneq}{\mathrel{\mathop:}\mkern-1.2mu=}
\nc{\margin}[1]{\marginpar{\scriptsize #1}}
\nc{\para}[1]{\bigskip\noindent\textbf{#1}}

\begin{document}
\title{(Un)distorted stabilisers in the handlebody group}
\author{Sebastian Hensel}
\date{\today}
\begin{abstract}
  We study geometric properties of stabilisers in the handlebody
  group. We find that stabilisers of meridians are
  undistorted, while stabilisers of primitive curves or annuli are exponentially
  distorted for large enough genus. 
\end{abstract}
\maketitle

\section{Introduction}
\label{sec:intro}
The handlebody group $\mathrm{Mcg}(V)$ is the mapping class group of a $3$--di\-men\-sional handlebody $V$. In this
article, we study the subgroup geometry of stabilisers in $\mathrm{Mcg}(V)$ of 
meridians and primitive curves. 
A curve $\delta$ on the boundary of a handlebody is called a \emph{meridian}, if it is the boundary of an embedded disc. A curve $\beta$ is called \emph{primitive}, if there is a meridian $\delta$ which intersects $\beta$ in a single point. 

Recall that a finitely generated subgroup $H < G$ of a finitely generated group $G$ is
\emph{undistorted} if the inclusion homomorphism is a quasi-isometric embedding. In contrast, we say
that it is \emph{exponentially distorted}, if the word norm in $H$ can be bounded by an exponential function of word norm in $G$, and there is no such bound of sub-exponential growth type. We refer the reader to e.g. \cite{Farb} for details on distortion functions of subgroups.


Our main results are:
\begin{theorem}\label{thm:intro-undistorted}
Suppose that $V$ is a handlebody or compression body, and that $\delta$ is a (multi)meridian in $V$. Then the stabiliser of $\delta$ 
is undistorted in $\mathrm{Mcg}(V)$.
\end{theorem}
\begin{theorem}\label{thm:intro-distorted}
Let $V$ be a handlebody of genus $g$.
\begin{enumerate}[i)]
\item Suppose that $\alpha$ is a primitive curve and $g \geq 3$. Then the stabiliser of $\alpha$ is exponentially distorted in $\mathrm{Mcg}(V)$.

\item Suppose $A \subset V$ is a properly embedded annulus so that $\partial A$ consists of primitive curves. Assume that $g\geq 3$ (if $A$ is non-separating) or $g\geq 4$ (if $A$ is separating). Then the stabiliser of $A$ is
exponentially distorted in $\mathrm{Mcg}(V)$.
\end{enumerate}
\end{theorem}
To put these theorems into context, observe that the handlebody group is directly related to mapping class groups of surfaces (via restriction of homeomorphisms to the boundary) and to the outer automorphism group of free groups (via the action on the fundamental group). 
However, neither of these connections is immediately useful to study the geometry of $\mathrm{Mcg}(V)$: The inclusion $\mathrm{Mcg}(V)\to \mathrm{Mcg}(\partial V)$ may distort distances exponentially \cite{HH1}, and the kernel
of the map $\mathrm{Mcg}(V)\to\mathrm{Out}(\pi_1(V))$ has an infinitely generated kernel \cite{Luft, McCullough}.
In other words, there is no a-priori reason to expect that $\mathrm{Mcg}(V)$ shares geometric features with
surface mapping class groups or outer automorphism groups of free groups.

It seems that its geometry, nevertheless, resembles that of outer automorphism groups of free groups.
A first instance of this was the computation of its Dehn function in \cite{HH2}: these are exponential for handlebody groups of genus at least three, just like those of $\mathrm{Out}(F_n)$ for $n\geq 3$.

The two main results of this article provide further evidence for this philosophy. In the
surface mapping class group $\mathrm{Mcg}(\partial V)$, stabilisers of curves are undistorted for all curves (this follows e.g. immediately from the distance formula of Masur-Minsky \cite{MM2}). 
On the other hand, Handel and Mosher \cite{HM} found that in the outer automorphism groups of
free groups there is a dichotomy -- stabilisers of free splittings are undistorted,
whereas stabilisers of primitive conjugacy classes (and most other free factors) are exponentially distorted.
By the van-Kampen theorem, the stabiliser of a meridian maps exactly to the stabiliser of a free splitting, while the stabiliser of a primitive curve maps to the stabiliser of a primitive conjugacy class.

\smallskip A properly embedded annulus in the handlebody induces, by Seifert-van Kampen, a splitting
of $\pi_1(V)$, amalgamated about a cyclic subgroup generated by a primitive element of $\pi_1(V)$.
We call such a splitting a \emph{primitive cyclic splitting}. Hence,
the stabiliser of an annulus in the handlebody
group maps to the stabiliser of a primitive cyclic splitting 
in $\mathrm{Out}(F_n)$. Here, 
an exponential lower bound for distortion follows from the results
in \cite{HM} (although it is not explicitly discussed in that reference).
We extend this by a similar upper bound, which
shows that it has the same behaviour as in the handlebody group case.
\begin{proposition}\label{prop:intro-cyclic-splitting}
	For $n\geq 4$, the stabiliser of a primitive cyclic splitting in $\mathrm{Out}(F_n)$ is exponentially distorted in $\mathrm{Out}(F_n)$.
\end{proposition}
%

\medskip To prove Theorem~\ref{thm:intro-undistorted}, we define a projection of meridian systems of $V$ to
meridian systems in sub-handlebodies. This is carried out in Section~\ref{sec:good-stabs}. We recall the well-known
fact that usual Masur-Minsky subsurface projections of meridians to sub-handlebodies are usually not meridians
again (compare e.g. \cite[Section~10]{H-Survey}), and so our projection procedure is more involved and depends on choices. The lower bounds in Theorem~\ref{thm:intro-distorted}
are proved by a reduction to the theorem by Handel-Mosher on stabilisers in the outer automorphism group of free groups. Here, the key difficulty is to realise certain free group automorphisms considered by Handel-Mosher as
homeomorphisms of handlebodies with comparable word norm. This uses an idea which was already employed in \cite{HH2}. 
The upper distortion bounds in Theorem~\ref{thm:intro-distorted} follow from a surgery procedure.

\medskip
\textbf{Acknowledgements} 
The author would like to thank Ursula Hamenst\"adt, Andy Putman, Saul Schleimer, and Ric Wade for interesting discussions
about stabilisers in different contexts. The author would also like the anonymous referees for numerous helpful commets. 
Finally, he would like to thank the mathematical institute of the University of Oxford, where part of this work was carried out. 

\section{Preliminaries}
\label{sec:prereqs}
In this section we collect some basic tools on surfaces, handlebodies, and compression bodies that
we will use throughout. 

\subsection{Surface Basics}
Suppose that $S$ is a surface. All \emph{curves} are assumed to be simple and essential.
We usually identify a curve with the isotopy class it defines. If $\alpha, \beta$ are two curves,
we denote by $i(\alpha, \beta)$ the \emph{geometric intersection number}, i.e. the smallest number of
intersections between $\alpha, \beta$ up to isotopy.

A \emph{subsurface} will always mean a (possibly disconnected) two-dimensional submanifold with boundary,
no component of which is an annulus. Unless specified explicitly, we will assume that all curves and 
subsurfaces are in minimal position with respect to each other; compare \cite{Primer}. 
If $\alpha$ is a curve and $Y$ is
a subsurface, then we call the components of $Y \cap \alpha$ the \emph{$\alpha$-arcs with
respect to $Y$}.

If $\delta$ is a multicurve, then we denote (by slight abuse of notation) by $S-\delta$ the complementary subsurface of $\delta$, i.e. the metric closure of $S\setminus\delta$ with respect to the path metric. Explicitly, this means
that $S-\delta$ has boundary components corresponding to the sides of the curves in $\delta$. Equivalently, one
can think of $S-\delta$ as the complement of a small open, regular neighbourhood of $\delta$.

We will often call the components of the intersection $\alpha\cap(S-\delta)$ which are not closed curves the \emph{$\alpha$-arcs with respect to $\delta$}. 

If $\alpha$ is a curve which intersects $\delta$ transversely, then we denote by
\[ \pi_{S-\delta}(\alpha) \] 
the \emph{subsurface projection}, which we define to be a maximal subset of non-homotopic arcs in the
set $(S-\delta)\cap \alpha$ of $\alpha$-arcs of $\delta$.
We define $\pi_{S-\delta}(\alpha) = \alpha$ if $\delta,\alpha$ are disjoint.

If $A$ is a multicurve, the projection $\pi_{S-\delta}(A)$ is defined to be the union of the projections
of components:
\[\pi_{S-\delta}(A) = \bigcup_{\alpha\in A}\pi_{S-\delta}(\alpha).\]
We remark that the subsurface projection of a multicurve in general is a union of (isotopy classes) of 
curves and proper arcs.

\subsection{Handlebodies and Surgeries}
A \emph{handlebody} is the $3$-manifold obtained by attaching three-dimensional one-handles to a $3$--ball. A \emph{compression body} is the
$3$--manifold obtained from a handlebody by taking boundary connect sums with a finite number of
trivial interval bundles over surfaces. A compression body has an \emph{outer boundary component}, which
is the unique boundary component of highest genus.

Suppose that $V$ is a handlebody or compression body. A \emph{meridian} is a simple closed
curve on $\partial V$ (resp. the outer boundary component of the compression body) which bounds a disc in $V$. A \emph{filling meridian system} is
a collection $\Delta = \{\delta_1, \ldots, \delta_k\}$ of disjoint, non-isotopic meridians, bounding disjoint
discs $D_i$ with the property that each component of $V-\cup D_i$ is either a $3$--ball or
a trivial interval bundle. Note that as handlebodies and compression bodies are aspherical, the disjointness
of the discs $D_i$ can always be arranged, if we suppose that the $\delta_i$ are disjoint.
 We remark that the number of components of a filling meridian system is non-unique, but
bounded above by $3g-3$. 

For handlebodies a filling meridian system contains at least $g$ curves.
For a handlebody $V$, we also use the notion of \emph{cut system}, by which we mean a filling meridian system
with the smallest number of elements. Equivalently, it is a filling meridian system with a single complementary
component (which is then necessarily a $2g$--holed sphere).

Next, we describe \emph{surgery}. Namely, we recall the following standard result 
(compare e.g. the proof of Theorem~5.3 in \cite{McC-geom-finite}).
\begin{lemma}\label{lem:surgery-basics}
	Suppose that $\Delta$ is a filling meridian system, and that $\alpha$ is any meridian
	which is not disjoint from $\Delta$.
	Then there is a subarc $a \subset \alpha$, called a \emph{wave (of $\alpha$ with respect to $\delta$)}, with the following
	properties:
	\begin{enumerate}[i)]
		\item $a \cap \Delta$ consists of two points on the same curve $\delta \in \Delta$. Call the components of $\delta-a = \delta_- \cup \delta_+$. Then the set
			\[ \{ a\cup\delta_-, a\cup\delta_+\}\cup \Delta\setminus\{\delta\} \]
		is a filling meridian system.
		\item If $V$ is a handlebody, then there is $\delta_*=\delta_\pm$ so that
			\[ \{ a\cup\delta_*\}\cup \Delta\setminus\{\delta\} \]
		is a filling meridian system. If $\Delta$ is a cut system, then exactly one choice of $\delta_\pm$ 
		yields a filling meridian system, and it is again a cut system.
	\end{enumerate}
	In the case of handlebodies, any arc $a$ which intersects $\Delta$ only in its endpoints
	and returns to the same side of a curve in $\Delta$ is a wave. Furthermore, in the case of 
	handelbodies there are always two distinct (but possibly homotopic) waves.
\end{lemma}
We call the result of i) \emph{full surgery}, and the result of ii) \emph{surgery of $\Delta$ in the direction of $\alpha$}. If $\Delta$ and the surgery of $\Delta$ are cut systems as in ii), we also call the process
\emph{cut system surgery}.

If $\alpha, \Delta$ are disjoint, we also say that $\Delta$ is obtained from $\Delta$ by surgery in the direction of  $\alpha$ (as this makes certain arguments later easier to state without case distinctions).

\smallskip
If $V$ is a handlebody, we let $\mathrm{Mcg}(V)$ be its mapping class group,
i.e. the group of orientation preserving self-homeomorphisms of $V$ up to isotopy. Since such homeomorphisms
and isotopies preserve the boundary, there is a restriction map
\[ \mathrm{Mcg}(V) \to \mathrm{Mcg}(\partial V), \]
where $\partial V$ is the boundary surface. We define the \emph{handlebody group} $\mathcal{H}(V)$ 
as the image of this map. In other words, the handlebody group consists of those mapping classes
of the surface $\partial V$, which can be extended to homeomorphisms of $V$. It is well-known that
the restriction map is injective, and so $\mathcal{H}(V)$ is in fact isomorphic to the mapping class
group $\mathrm{Mcg}(V)$. However, for us it will be more convenient to think of elements in the 
handlebody group as surface mapping classes. 

If $S \subset \partial V$ is a subsurface, then 
we denote by $\mathcal{H}(S)$ the subgroup of $\mathcal{H}(V)$ formed by all elements supported
in $S$.

\smallskip The following well-known lemma is central for us:

\begin{lemma}[{e.g. \cite[Corollary~5.11]{H-Survey}}]\label{lem:han-criterion}
	Suppose that $\varphi:\partial V\to \partial V$ is a mapping class. If both $C$ 
	and $f(C)$ are filling meridian systems, then $\varphi$ is contained in the handlebody group.
\end{lemma}

We define the \emph{compression body group} similarly, replacing $\partial V$ with the outer boundary
component of the compression body $V$.


\section{Meridian Stabilisers}
\label{sec:good-stabs}
In this section we analyse stabilisers of meridians in handlebody and compression
body groups. For simplicity of exposition we focus on the proof in the case
of handlebody groups, and only indicate the necessary modifications in the case
of compression body groups at the end of the section. 

From an algebraic perspective, the study of meridian stabilisers reduces to the 
study of point-pushing and handlebody groups of smaller genus, just as in the case
of surface mapping class groups; compare \cite[Section~3]{H-Survey} for details.
In particular, this implies that multimeridian stabilisers are finitely generated.

The main result of this section is the following theorem:
\begin{theorem}\label{thm:undistorted-stabilisers}
  Let $V$ be a handlebody, and $\delta$ be a multimeridian. Then the
  stabiliser of $\delta$ in the handlebody group is undistorted.
\end{theorem}

%

\subsection{Distortion and Models}
In this section, we describe the main mechanism by which we will exhibit (un-)distortion. This is
the following relative version of the \v{S}varc-Milnor lemma. While this is well-known to experts, we provide a full formulation and proof for the benefit of the reader. If $X$ is a graph and $Y\subset X$ is a connected subgraph,
we define the distortion function as
\[ D_Y(n) = \max \{d_Y(v,w) \vert v,w\in Y^0, d_X(v,w)\leq n. \} \]
Here $d_X, d_Y$ are the path-metrics on $X,Y$ obtained by declaring each edge to have length $1$.
It is well-known (and not hard to see) that the growth type of $D_Y$ is invariant under quasi-isometries of the
pair $(X,Y)$. For finitely subgroups $H<G$ of finitely generated groups, we define the distortion function 
of the subgroup $H$ as the distortion function for (any) inclusion of Cayley graphs.

\begin{lemma}\label{lem:group-to-model}
	Suppose that $G$ is a group, and $H$ is a subgroup of $G$. Assume that $\mathcal{G}$ is a connected, locally finite graph, on which $G$ acts by isometries and with finite quotient
	and finite stabilisers.
	
	If $\mathcal{G}_H \subset \mathcal{G}$ is a connected subgraph, which is preserved by $H$, and so that
	$\mathcal{G}_H/H$ is finite, then the distortion function of $H$ in $G$ has the same growth type as the
	distortion function of $\mathcal{G}_H$ in $\mathcal{G}$.
\end{lemma}
\begin{proof}
	Pick a basepoint $o \in \mathcal{G}_H$. By the usual \v{S}varc-Milnor lemma, the orbit maps
	\[ H \to \mathcal{G}_H, \quad h \mapsto h\cdot o \]
	and
	\[ G \to \mathcal{G}, \quad g \mapsto g\cdot o\]
	are quasi-isometric embeddings. This shows that the pair $(\mathrm{Cay}(G), \mathrm{Cay}(H))$
	is quasi-isometric to the pair $(\mathcal{G}, \mathcal{G}_H)$, which shows the lemma.
\end{proof}

We will use two geometric models for the handlebody group in the proof of Theorem~\ref{thm:undistorted-stabilisers}
(and subsequent arguments). For the definition of the first, we need the notion of a \emph{discbusting curve}:
we say that $l$ is discbusting, if $l$ is not disjoint (up to isotopy) to any meridian.
\begin{definition}
  For numbers $k_0, k_1>0$, define a graph $\mathcal{G}(k_0, k_1)$ with
  \begin{description}
  \item[Vertices] corresponding to pairs $(C,l)$ of a filling meridian
    system $C$ and a simple discbusting loop $l$ (up to isotopy), so that
    $i(C, l) \leq k_0$.
  \item[C-Edges] between vertices $(C, l)$ and $(C', l)$ if $C'$ is disjoint
    from $C$.
  \item[l-Edges] between vertices $(C, l)$ and $(C, l')$ if
    $i(l, l')\leq k_1$.
  \end{description}
\end{definition}
The following is a standard trick to obtain a geometric model for a group; compare e.g.~\cite[Lemma~7.3]{HH1}.
\begin{lemma}\label{lem:modelbuilding}
	There are choices of $k_0, k_1$ so that $\mathcal{G}(k_0, k_1)$ is nonempty, connected, locally finite, and the handlebody group acts on $\mathcal{G}(k_0, k_1)$ properly discontinuously and cocompactly.
\end{lemma}
\begin{proof}
	First, observe that there are finitely many filling meridian systems $C_1, \ldots C_k$ so that
	every filling meridian system is in the handlebody group orbit of one of the $C_i$. This 
	follows from Lemma~\ref{lem:han-criterion} and the fact that there are only finitely many types of multicurves up to the action of 
	the mapping class group. As being discbusting is invariant under the action of the handlebody group,
	this shows that there is a constant $k_0$ so that for any filling meridian system $C$ there is a discbusting
	loop $l$ so that $i(C, l) \leq k_0$.
	
	\smallskip Next, we claim that the stabiliser of $C$ in the handlebody group acts with finitly many orbits on the set of (isotopy classes)
	of pairs $(C,l)$ as in the definition of the vertices. By the intersection number bound, $l \cap (S-C)$ is a collection of at most
	$k_0$ embedded arcs. Up to the action of the mapping class group of $S-C$, there are only finitely many
	such arc systems. By Lemma~\ref{lem:han-criterion}, the same is therefore true for the action of the stabiliser of $C$ in the handlebody group.
	Up to Dehn twists about the curves in $C$, there are only finitely many possible curves $l$ which can be
	obtained by connecting the arcs $l \cap (S-C)$. Since the Dehn twists about $C$ are contained in the
	handlebody group as well, this implies the claim.
	
	\smallskip As a consequence, we first claim that we can choose the constant $k_1$ large enough so that any two vertices of the form $(C, l), (C, l')$ are connected by a path of l-edges. To see this, fix a set $l_1, \ldots, l_r$ of orbit representatives of such discbusting curves (for the action of the stabiliser of $C$), and a finite generating set $\varphi_1, \ldots, \varphi_k$ of the stabiliser of $C$ in the handlebody group. Now, 
	if $k_1 > \max_{a,b,c} i(l_a,\varphi_bl_c)$ the desired property holds.
	
	\smallskip Now, suppose that $(C,l)$ and $(C', l')$ are arbitrary. We claim that there is a path between them.
	To show this, first find a sequence $C = C_0, C_1, \ldots, C_N = C'$ of filling meridian systems so that 
	$C_i, C_{i+1}$ are disjoint. This is possible e.g. by surgery. Next, find discbusting curves $l_i$ so 
	that $(C_i, l_i), (C_{i+1}, l_i)$ represent vertices of the graph. This is possible by the choice of $k_0$ above.
	In particular, $(C_i, l_i), (C_{i+1}, l_i)$ are therefore joined by a C-edge.
	From this, we can assemble the desired path as
	\[ (C_0,l) \to (C_0,l_0) \to (C_1, l_0) \to (C_1, l_1) \to (C_2, l_1) \to \cdots \to (C_N,l_{N-1}) \to (C_N, l') \]
	where every $\to$ denotes a l-edge or C-edge. This shows that the graph is connected.
	
	\smallskip Next, observe that the graph is locally finite. This follows since for any vertex $(C,l)$ the curve
	system $C\cup \{l\}$ is filling, and there are therefore only finitely many isotopy classes of curves
	whose intersection number with $C\cup\{l\}$ is bounded by a given constant. By the same reason, the stabiliser of any vertex is finite (indeed, the stabiliser of a filling curve system in the mapping class group is finite).
	Thus, the action of the handlebody group is properly discontinous.
	
	Cocompactness of the action then follows immediately, since (by construction) there are only finitely many
	orbits of vertices.
\end{proof}
In the sequel, we make choice of $k_0, k_1$ as in the lemma, and simply denote the resulting
graph by $\mathcal{G}$. Observer that by the \v{S}varc-Milnor lemma, $\mathcal{G}$ 
will then be equivariantly quasi-isometric to the handlebody group.

\begin{definition}
  For a multimeridian $\delta$, we let $\mathcal{G}(\delta)$ be the full subgraph
  spanned by all those vertices $(C,l)$ whose meridian system contains $\delta$.
\end{definition}
\begin{corollary}\label{cor:modelbuilding}
	The subgraphs $\mathcal{G}(\delta)$ are connected for all $\delta$, and
	the stabiliser of $\delta$ in the handlebody group acts on $\mathcal{G}(\delta)$ cocompactly.
\end{corollary}
\begin{proof}
	This is a corollary of the proof of Lemma~\ref{lem:modelbuilding}. First, for connectivity, we just
	have to observe that the sequence $C_0, \ldots, C_N$ can be chosen to contain $\delta$. This is 
	automatic for surgery sequences (if both $C, C'$ contain $\delta$, then every step of the surgery
	sequence will have the same property).
	
	For cocompactness, one observes that the stabiliser of $\delta$ acts with finitely many orbits on
	the set of filling meridian systems which contain $\delta$. One way to see this is to note that
	the stabiliser of a multicurve $\delta$ in the mapping class group acts with finitely many orbits
	on the set of multicurves containing $\delta$, and using again Lemma~\ref{lem:han-criterion}.
\end{proof}
Now, Lemma~\ref{lem:group-to-model} implies that in order to prove
Theorem~\ref{thm:undistorted-stabilisers}, it suffices to show that
the subgraph $\mathcal{G}(\delta)$ is undistorted in $\mathcal{G}$.

In the proof it will be useful to use a second model\footnote{The reason for using both models is
  that curves and curve pairs have easier to phrase minimal position properties
  as opposed to embedded graphs.}, which is 
similar to the graph of rigid racks employed in~\cite{HH1}.
\begin{definition}
  For numbers $k_0, k_1$, the graph $\mathcal{R}(k_0, k_1)$ has
  \begin{description}
  \item[Vertices] corresponding to (isotopy classes of) connected graphs
    $\Gamma\subset\partial V$ with at most $k_0$ vertices, which
    contain a filling meridian system $C(\Gamma)$ as an embedded subgraph, and have simply connected complementary regions.
  \item[Edges] between graphs $\Gamma, \Gamma'$ which intersect in at most $k_1$ points (up to isotopy). 
  \end{description}
\end{definition}
Arguing as in the proof of Lemma~\ref{lem:modelbuilding}, we can choose the constants $k_0, k_1$ so that the resulting graph is connected, locally finite, and the action of the handlebody group is properly discontinuous and cocompact (here, the role of $l$ is played by the complement of the embeddded subgraph defined by the filling meridian system). Similar to the definition of $\mathcal{G}(\delta)$, we define
$\mathcal{R}(\delta)$ to be the full subgraph spanned by vertices corresponding to graphs $\Gamma$, to that the 
filling meridian system $C(\Gamma)$ also contains $\delta$. Arguing as above, we may assume that it is also connected, and the stabiliser of $\delta$ acts cocompactly.

By possibly increasing the constants defining $\mathcal{R}$, 
there is a natural map
\[ U:\mathcal{G} \to \mathcal{R} \] which sends a vertex $(C, l)$ to the
union $C\cup l$ (assuming that $l$ is in minimal position with respect
to $C$). Observe that the map $U$ is equivariant for the action of the handlebody group
on $\mathcal{G}, \mathcal{R}$. As a consequence, $U$ is a quasi-isometry, as the handlebody
group acts cocompactly on both of these graphs. The restriction
\[ U:\mathcal{G}(\delta) \to \mathcal{R}(\delta) \] is also a quasi-isometry,
as it is equivariant for the stabiliser of $\delta$, which acts cocompactly on both sides.

\smallskip The strategy to
prove Theorem~\ref{thm:undistorted-stabilisers} is to start with any
path $\gamma:[0, n]\to\mathcal{G}$, and to ``project'' it to a path in
$\mathcal{R}(\delta)$ joining $U(\gamma(0))$ to $U(\gamma(n))$, taking care  that the length of the 
projected path is coarsely bounded by $n$. Since $U$ is a quasi-isometry, this
will imply that $\mathcal{G}(\delta)$ is undistorted in $\mathcal{G}$,
showing Theorem~\ref{thm:undistorted-stabilisers}.

\subsection{Patterns and Surgery}
In this section we will describe a systematic way to simplify a cut
system until it is disjoint from a given (multi)meridian $\delta$,
which is assumed to be fixed throughout the section. This will be the core 
ingredient used to project paths in $\mathcal{G}$ to $\mathcal{R}(\delta)$.
Recall that a cut system is a collection of $g$ meridians $\alpha_1, \ldots, \alpha_g$
the complement of which is a $2g$--holed sphere.
 We could
also perform these arguments with general filling meridian systems, but the
discussion is slightly more convenient in the cut system case.

We use
following terminology and setup throughout this section. Cut the boundary
surface $\partial V$ of the handlebody at the meridian $\delta$. The
resulting (possibly disconnected) surface with boundary has boundary
components $\delta^+, \delta^-$ corresponding to the two sides of
$\delta$.

Suppose now that $C$ is a cut system intersecting $\delta$ transversely 
and minimally. 
Then we call, by a slight abuse of notation, the set
\[ (\delta^+ \cap C) \cup (\delta^-\cap C) \] the \emph{intersection points of $C$
with $\delta$}. Recall that we call the connected components of $C\cap (S-\delta)$
the \emph{$C$--arcs}. Note that endpoints of $C$-arcs are precisely the intersection
points of $C$ with $\delta$.

An \emph{interval} will mean a subarc $I$ of $\delta^+\cup\delta^-$
whose endpoints lie in $C$. 

Observe that for the two endpoints $x,y$
there are uniquely determined $C$--arcs $\gamma_x, \gamma_y$ which
intersect $I$ in $x,y$ (note that these arcs may coincide). We call them
\emph{$C$--arcs adjacent to $I$}.

\begin{definition}
  A \emph{partial pattern} for $C$ is a collection $\mathcal{I}$ of
  intervals satisfying the following properties:
  \begin{description}
  \item[N)] Any two $I, J \in \mathcal{I}$ are disjoint or nested (i.e. one is contained in the 
  interior of the other).
  \item[P)] If $\gamma$ is a $C$--arc adjacent to some $I \in \mathcal{I}$, then both
    endpoints of $\gamma$ are endpoints of intervals in $\mathcal{I}$.
  \end{description}
\end{definition}
\begin{definition}
  A \emph{chain} of a partial pattern $\mathcal{I}$ is a sequence
  \[c=(I_1,\gamma_1,I_2,\gamma_2,\ldots,\gamma_k, I_1)\] of intervals in $\mathcal{I}$ 
  and $C$--arcs so that 
  \begin{enumerate}[i)]
  	\item the intervals $I_j$ can be oriented so that for each $i$, the right endpoint of $I_i$ is
  	one endpoint of $\gamma_i$, and the other endpoint of $\gamma_i$ is
  	the left endpoint of $I_{i+1}$, and
  	\item the sequence does not contain a subsequence satisfying i).
  \end{enumerate}
\end{definition}
An easy induction, using property P), shows that every interval in a
partial pattern is part of a chain which is unique up to cyclic reordering.  By concatenating the $I_i$ and
$\gamma_i$ in a chain $c$, we obtain a closed loop on the surface,
which by abuse of notation we also call a chain of the partial pattern. 

This
concatenation is not embedded, but we can homotope it into \emph{push-off
  position}, by pushing the intervals $I_i$ slightly off of
$\delta^+\cup\delta^-$ into $S-\delta$.  Property N) guarantees that
we can choose push-off positions for all chains so that the chains
themselves are simple closed curves, and different chains do not
intersect (push off more deeply nested intervals further off of
$\delta^+\cup\delta^-$).  From now on we assume that, unless specified
explicitly, all chains are in push-off position.

A chain may be an inessential curve. To avoid this, we put 
\[ C(\mathcal{I}) = \{ \alpha \vert \alpha\mbox{ is an essential curve
  defined by a chain of }\mathcal{I}\}. \]
Each $c\in C(\mathcal{I})$ is a concatenation of subarcs of $C$
and (pushed off copies of) intervals in $\mathcal{I}$.  We call the
intervals $I\in \mathcal{I}$ which appear in this way \emph{active} (intervals may be inactive
if they lie on inessential chains defined by the pattern).
\begin{definition}
  A \emph{pattern} is a partial pattern $\mathcal{I}$ which
  additionally satisfies:
  \begin{description}
  \item[F)] The set $C(\mathcal{I})\cup\{\delta\}$ forms a filling
    meridian system for $V$.
  \end{description}
\end{definition}
\begin{definition}[Compatible Patterns]
  If $C'$ is disjoint from $C$, then a pattern $\mathcal{I}$ for $C$
  and a pattern $\mathcal{I}'$ for $C'$ are \emph{compatible}, if any
  intervals $I\in\mathcal{I}$ and $I'\in\mathcal{I}'$ are either disjoint
  or nested.
\end{definition}
Arguing as before with push-off representatives, we immediately obtain
the following:
\begin{lemma}
  If $C, C'$ are disjoint cut systems, and $\mathcal{I}, \mathcal{I}'$
  are compatible patterns, then $C(\mathcal{I})$ and
  $C'(\mathcal{I'})$ are disjoint.
\end{lemma}

Our first goal is to show that patterns always exist, and that for
disjoint cut systems there are compatible patterns. This will be done
by a standard surgery procedure, and the following lemma is analogous
to various results in the literature, compare e.g.~
\cite[Lemma~5.4]{HH1}, \cite[Lemma~1.3]{Hempel} or
\cite[Lemma~1.1]{Masur}.
\begin{lemma}[(Compatible) patterns exist]\label{lem:generating-patterns}
  \begin{enumerate}[i)]
  \item For any filling meridian system $C$ in minimal position with respect to a multimeridian $\delta$, there is a
    pattern $\mathcal{I}$ for $C$, which we call a \emph{surgery pattern}.
  \item If $\mathcal{I}$ is a 
    surgery pattern (i.e. a pattern obtained by applying part i) of this lemma) for $C$, and $C'$ is
    disjoint from $C$, then there is a 
    surgery pattern for $C'$ which is compatible with $\mathcal{I}$.
  \end{enumerate}
\end{lemma}
\begin{proof}
  \begin{enumerate}[i)]
  \item We find the pattern inductively. Note that if an interval $I$
    is innermost, i.e. it does not contain any other intersection
    points of $C$ with $\delta$, then it defines an $\delta$--arc with
    respect to $C$. Thus it makes sense to talk about intervals being a wave of $\delta$
    with respect to $C$.
    Choose an interval $I_1$ which is a wave $w$
    of $\delta$ with respect to $C$ and set $C=C_1$. 
    Let $C_2$ be the result of the cut system surgery of $C_1$ at $w$ (i.e. the surgery of $C_1$ at $w$
    so that the result is a cut system). In fact, 
    $C_2$ has a push-off representative as well, so that each curve in $C_2$ is a concatenation
    of parts of $C_1$ and the interval $I_1$ (compare Figure~\ref{fig:pushoff-surgery}).
    \begin{figure}
      \centering
      \includegraphics[width=0.8\textwidth]{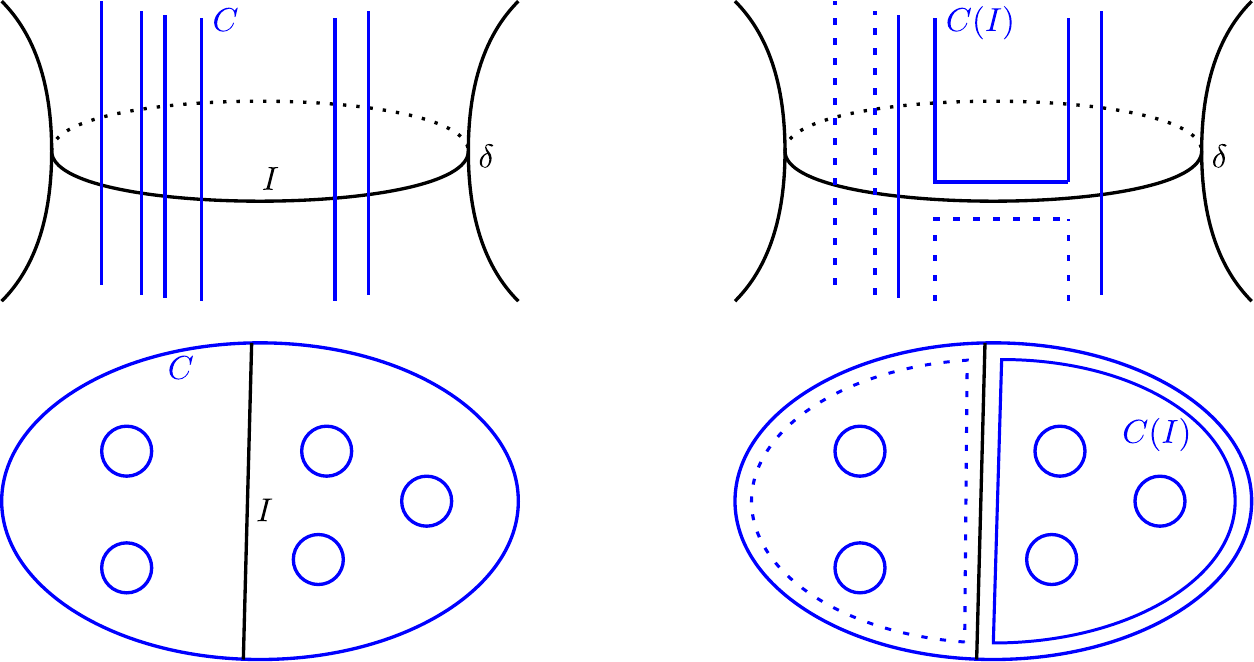}
      \caption{Push-off representatives of a surgery move.}
      \label{fig:pushoff-surgery}
    \end{figure}
	It might be the case that this push-off-representative is not in minimal position with respect 
	to $\delta$ (if there are other waves isotopic to $w$), but there is a choice of $I_1$ so that
	this problem does not occur (by choosing $I_1$ outermost in the sense that in the direction of the push-off there is
	no parallel copy of $w$).

    Now inductively repeat this procedure, defining a surgery
    sequence $C_i$ and a sequence of intervals $I_i$, so that each
    curve in each $C_i$ is obtained as a concatenation of arcs in $C$
    and intervals $I_k, k \leq i$. The final term $C_n$ of this
    surgery sequence is a cut system disjoint from $\delta$. 
	Let $\mathcal{I}$ be the set of those intervals $I_k$ which are still part of $C_n$.
    By construction, we then have that $C(\mathcal{I}) = C_n$.
    This implies that $\mathcal{I}$ is indeed a pattern. Namely, following along the procedure we know that 
    $C_n$ is the push-off representative of $C(\mathcal{I})$. The fact that each component
    of $C(\mathcal{I})$ closes up to be a multicurve implies that $\mathcal{I}$ has property P),
    and the fact that $C(\mathcal{I})$ is simple implies it has property N), 
  	as, if this were not the case, push-off-representatives would have self-intersections. Since $C_n$ is
  	a filling meridian system, $\mathcal{I}$ is indeed a pattern.
 
  \item 
  	Suppose $\mathcal{I}$ is a surgery pattern. Recall that this means that $\mathcal{I}$ is formed 
  	by the construction of part i), and let $C_i, I_i$ denote the sequences constructed there. 
  	To define the desired pattern $\mathcal{I}'$, we will follow a similar construction as in part i), 
	successively building intervals $I'_i$ and meridian systems $C'_i$, using the $C_i, I_i$.
	In each step of the procedure, we will either build a new interval $I_j'$ and a new meridian system  $C'_j$, or	we will show that the current interval is compatible with the next corresponding interval in $\mathcal{I}$.
	
	We begin by putting $C'_1 = C'$. By assumption, $C_1$ is then disjoint from $C'_1$.
	Consider the interval $I_1$: if $I_1$ is disjoint from $C'_1$, then so is $C_2=C_1(I_1)$.
	If $I_1$ intersects $C'_1$, then there is a subinterval $I'_1\subset I_1$ with endpoints
	on $C'_1$ which is also a wave of $\delta$ with respect to $C'_1$. In that case, note
	that $I'_1$ is disjoint from $C_1$, and therefore the same is true for $C'_2 := C'_1(I'_1)$.
	
	Now suppose that $I'_j, j<s, C'_s, j\leq s$ are defined for some $s$, and that there is a
	also a number $r$ with the property that so that $I_i,
    i\leq r, I'_j, j < s$ are nested or disjoint, and so that $C_r,
    C'_s$ are disjoint. In the first case of the previous case distinction, we have $s=0, r=1$ or
    $s=1, r=0$ respectively.
    
    Now consider $I_r$, and again distinguish two
    cases. If $I_r$ is disjoint from $C_s'$, then $C_{r+1}$ is also disjoint from 
    $C'_s$, and $I_r$ is nested or disjoint from the $I'_j, j \leq s$. Hence,
    $(r+1,s)$ again satisfies the assumption above.

    Alternatively, if $I_r$ is not disjoint from $C_s'$, then there is
    a sub-interval $I_r' \subset I_r$ which defines a wave for
    $C'_s$. In that case, let $C'_{s+1}$ be the surgery of $C'_s$ at
    that sub-interval, and note that it is disjoint from $C_r$. Hence,
    $(r, s+1)$ satisfies the assumption above. 
    
    After finitely many such steps, all surgery steps for the sequence $C_i$ have been
    taken, and we have constructed a partial pattern $\mathcal{I}'$ compatible with
    $\mathcal{I}$. Now, we can complete the construction as in i) to build the 
    pattern $\mathcal{I}'$ with the desired properties.
  \end{enumerate}
\end{proof}
The surgery patterns produced by Lemma~\ref{lem:generating-patterns}
are not yet sufficient for our purposes. We will need a second move to
improve patterns; to describe it we use the following terminology:

Suppose that $\mathcal{I}$ is a pattern for $C$, and suppose that $I
\subsetneq J$ are two active intervals in $\mathcal{I}$.  The two
segments of $J\setminus \mathrm{int}(I)$, call them $w_, w_+$ can be
interpreted (up to a small homotopy) as arcs with endpoints on
$C(\mathcal{I})$, and we call these arcs \emph{wings}; compare
Figure~\ref{fig:wings}. Observe that the two wings defined by a nested
pair of intervals are homotopic as arcs with endpoints sliding on $C(\mathcal{I})$.
\begin{figure}
  \centering
  \includegraphics[width=\textwidth]{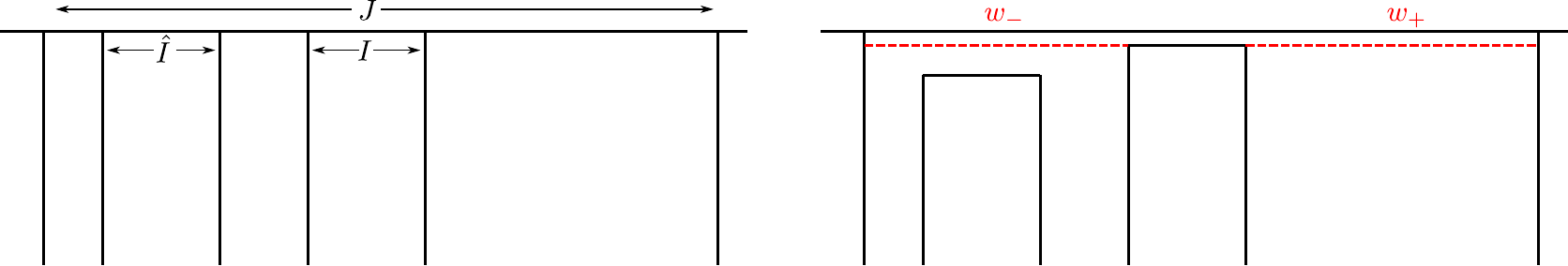}
  \caption{Nested intervals in a pattern define wings. If $I \subset
    J$ are directly nested, the wings can be made disjoint from the
    system $C(\mathcal{I})$, even if there are other intervals $\hat{I} \subset
    J$.}
  \label{fig:wings}
\end{figure}
We say that $I, J$ are \emph{directly nested} if there is no active
$I'$ with $I \subsetneq I' \subsetneq J$. Suppose now that $I
\subsetneq J$ are directly nested. In that case, wings define arcs
that intersect $C(\mathcal{I})$ only in their endpoints (compare again
Figure~\ref{fig:wings} for this situation), since intersections of a
wing with intervals $\hat{I}$ satisfying $\hat{I} \subset J-I$ can be
removed up to homotopy.

Suppose now that $D$ is any filling meridian system, and that $w$ is an embedded 
arc which has both endpoints on the same curve $\gamma\in D$. Then $w$
defines an element in $\pi_1(V)$ by connecting the endpoints of $w$ in
any way along $\gamma$ (observe that since $\gamma$ is a meridian, the resulting
element of $\pi_1(V)$ does not depend on the choice of interval in $\gamma$, but only on $w$). We say that $w$ is a
\emph{$V$-trivial arc} if it has both endpoints on the same curve
$\gamma$ and additionally it defines the trivial element in
$\pi_1(V)$.
\begin{lemma}\label{lem:trivial-arc-wave}
  If $D$ is a filling meridian system and $w$ is an embedded $V$--trivial arc with endpoints on
  a curve in $D$ (which may intersect $D$ in more points than just its endpoints), then
  $w$ contains a subarc $w_0 \subset w$ which defines a wave with
  respect to $D$.
\end{lemma}
\begin{proof}
  We use a particular covering space of $\partial V$. Namely, consider
  the universal covering $\widetilde{V} \to V$, and let $Y
  \to \partial V$ be the restriction to the boundary.  By the lifting
  theorem, a closed (possibly nonembedded) loop on $\partial V$ lifts to a closed (possibly nonembedded) loop on $Y$
  exactly if it defines the trivial element in $\pi_1(V)$. In other
  words, the cover $Y$ is the cover corresponding to the subgroup
  $\pi_1(Y) = \ker(\pi_1(\partial V)\to \pi_1(V))$.

  Every lift of a curve in $D$ to $Y$ is seperating (this can e.g. be
  seen by observing that $V$ is homotopy equivalent to a tree in a way so that
  each lift of a curve in $D$ maps to a point).

  Now, choose a lift $\widetilde{w}$ of $w$ to $Y$. Since $w$ is
  assumed to be $V$--trivial, this lift joins some lift
  $\widetilde{\delta}$ of a curve $\delta$ in $D$ to itself. Namely,
  choose a subarc $d\subset \delta \in D$ so that the
  concatenation $d*w$ is a (not necessarily embedded) loop in
  $V$ which is trivial in $\pi_1(V)$. Hence, $d*w$ lifts to a closed loop in
  $Y$ (which also need not be embedded), which has the form
  $\widetilde{d}*\widetilde{w}$, where $\widetilde{d}$ is contained in a lift
  $\widetilde{\delta}$ of the curve $\delta$. In particular, $\widetilde{w}$ has
  both endpoints on that lift.
  
  Since all lifts of curves in $D$ are separating in $Y$, and we assume that 
  $w$ (hence $\widetilde{w}$) is embedded, there is a subarc
  $\widetilde{w}_0 \subset \widetilde{w}$ whose interior is disjoint from
  all lifts of curves in $D$, and has both endpoints on the same lift. 


  The image $w_0$ of $\widetilde{w}_0$ under the covering map is a
  subarc of $w$ which has both endpoints on the same curve of
  $D$, and approaches it from the same side at both endpoints. By the
  last part of Lemma~\ref{lem:surgery-basics}, $w_0$ is then the
  desired wave.
\end{proof}

\begin{lemma}\label{lem:structure-wings}
  Let $C$ be a cut system and $\mathcal{I}$ a pattern.  Suppose that
  $I\subsetneq J$ are two active intervals of $\mathcal{I}$, and that
  $w$ is a wing of $I\subsetneq J$.
  \begin{enumerate}[i)]
  \item If $w$ is $V$-trivial, then there are directly nested 
    active intervals $I' \subsetneq J'$ which have a wing $w'\subset w$ that defines a wave.
  \item Suppose now that  $I\subsetneq J$ are directly nested, and that $w$ defines a wave.
   Let $C'$ be a filling meridian system which is disjoint from $C$ and let $\mathcal{I}'$
    be a pattern for $C'$ which is compatible with $\mathcal{I}$. Then either the wing $w$ can be
    made disjoint from $C'(\mathcal{I}')$, or there is a pair $I'\subsetneq J'$ of directly
    nested intervals in $\mathcal{I}'$ whose wing $w'$ is a wave and contained in $w$.
  \end{enumerate}
\end{lemma}
\begin{proof}
  \begin{enumerate}[i)]
  \item Consider $I=I_1 \subsetneq I_2 \cdots \subsetneq I_k = J$ a
    maximal chain of nested final intervals in the pattern. Without loss of generality we may assume
    that $w$ connects the left endpoint of $J$ to the left endpoint of $I$.
    We can push 
    $w$ off of the curve of $C$ containing it a little bit to form an arc $w'$, so that the intersections
    of $w'$ with $C(\mathcal{I})$ exactly correspond to the points $p \in w$ which
    are endpoints of intervals $K\in\mathcal{I}$. 
    
    If $K$ is an interval both endpoints of which are contained in $w$, then an isotopy
    of $C(\mathcal{I})$ can remove both of these intersections (compare Figure~\ref{fig:wings}).
    After this modification, the intersection points of $w'$ with 
    $C(\mathcal{I})$ exactly correspond to the left endpoints of the $I_j$. 
    
    Since $w'$ is $V$-trivial loop, its intersection with the filling meridian system
    $C(\mathcal{I}) \cup \{\delta\}$ has a
    wave by Lemma~\ref{lem:trivial-arc-wave}. The endpoints of the wave $w_0 \subset w'$  then correspond to 
    the left endpoints of $I_i, I_{i+1}$ for some $i$. This implies that $I_i, I_{i+1}$ are directly
    nested, and they have a wing which is a wave as claimed.
  \item Arguing as in i), we see that an essential intersection of the
    wing $w$ with the system $C'(\mathcal{I}')$ can only occur if
    there are intervals $I \subsetneq I' \subsetneq J$ with $I'\in
    \mathcal{I}'$. In fact, as above, the essential intersections of $w$ then exactly correspond to
    intervals $I_j'\in\mathcal{I}'$ with $I\subsetneq I_1' \subsetneq \cdots \subsetneq I_k' \subsetneq J$.
    Apply Lemma~\ref{lem:trivial-arc-wave} to $w$ and the filling meridian system $C(\mathcal{I})\cup C'(\mathcal{I}')\cup \{\delta\}$ as above. There is thus a wave $w_1 \subset w$ of this arc. Its endpoints
    are necessarily contained in $C'(\mathcal{I}')$ since the interior of $w$ is disjoint from $C(\mathcal{I})$ and $\delta$, and so $w_1$ defines the desired nested intervals as in i).
  \end{enumerate}
\end{proof}
Finally, we choose once and for all a hyperbolic metric on the surface
$S$. This allows us to talk about the length of intervals. The \emph{length}
of a pattern is the sum of the lengths of all intervals chosen by the pattern.
We can now describe the \emph{wave exchange move}. Suppose that
$\mathcal{I}$ is a pattern for some cut system $C$, and suppose that
$I\subsetneq J$ are directly nested active intervals whose wings $w_1,
w_2$ are waves. The wave exchange $\mathcal{I}(w_1, w_2)$ is the set
obtained by replacing $\{I, J\}$ by $\{w_1, w_2\}$; see Figure~\ref{fig:wave-flip} for 
this construction.
\begin{lemma}\label{lem:exchange-is-pattern}
  The wave exchange $\mathcal{I}(w_1, w_2)$ is a pattern of strictly
  smaller length than $\mathcal{I}$.
\end{lemma}
\begin{proof}
	We begin by noting that $\mathcal{I}(w_1, w_2)$ satisfies property N) since 
	the intervals $I\subsetneq J$ are supposed to be directly nested. Hence, any 
	interval in  $\mathcal{I}\setminus\{I,J\}$ which intersects $J-I$ is contained in $J-I$, and therefore
	nested inside $w_1$ or $w_2$. Property P) is obvious, since $\mathcal{I}$ and
	$\mathcal{I}(w_1, w_2)$ connect the same intersection points of $\delta$ with $C$.
	To show property F), note that $C(\mathcal{I}(w_1, w_2))$ is the full surgery
  of $C(\mathcal{I})$ at the wave $w_1$ (or $w_2$), and therefore is still filling
  (compare Figure~\ref{fig:wave-flip}). 
    \begin{figure}
      \centering
      \includegraphics[width=0.6\textwidth]{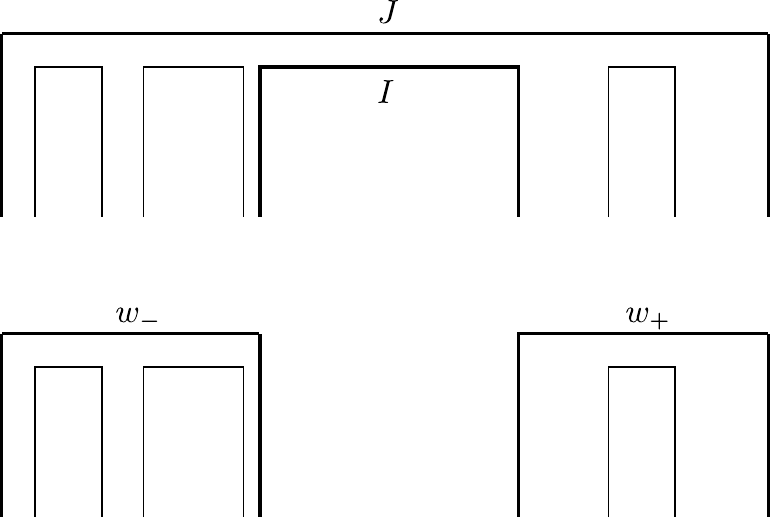}
      \caption{Suppose $I \subset J$ are directly nested, with wings
        defining waves, depicted above. Below is the wave exchange.}
      \label{fig:wave-flip}
    \end{figure}
  The claim about length is due to the fact that $w_1\cup w_2 \subsetneq J$.
\end{proof}
\begin{definition}
  Call a pattern \emph{essential} if no nested intervals have wings which are
  $V$-trivial arcs.
\end{definition}
\begin{corollary}\label{cor:wave-exchange-terminates}
  Any pattern $\mathcal{I}$ admits a finite sequence of waves exchanges
  terminating in an essential pattern.
\end{corollary}
\begin{proof}
	Since by the previous lemma each wave exchange strictly decreases
	the length of the pattern, any sequence of wave exchanges starting
	in $\mathcal{I}$ is finite, and terminates in a pattern $\mathcal{I}'$
	which does not admit any wave exchanges. Then by Lemma~\ref{lem:structure-wings}, 
	no nested pair of intervals of $\mathcal{I}'$ has a wing which defines a $V$-trivial 
	arc, as otherwise there would be a pair defining a wave, and hence the pattern would admit
  another wave exchange. Hence, $\mathcal{I}'$ is essential as claimed. 
\end{proof}
\begin{lemma}\label{lem:adjacent-exchange}
  Suppose that $\mathcal{I}$ is a pattern for some cut system $C$, and
  suppose that $I\subsetneq J$ are directly nested active intervals
  whose wings $w_1, w_2$ are waves. Suppose further that $C'$ is
  disjoint from $C$ and that $\mathcal{I}'$ is a compatible pattern.

  Then either $\mathcal{I}(w_1, w_2)$ is compatible with
  $\mathcal{I}'$, or there is a wave exchange $\mathcal{I}'(w'_1,
  w'_2)$ which is compatible with $\mathcal{I}$.
\end{lemma}
\begin{proof}
  If $w_1$ (and thus $w_2$) can be made disjoint from
  $C'(\mathcal{I}')$, then the wave exchange $\mathcal{I}(w_1, w_2)$
  is compatible with $\mathcal{I}'$. As in the proof of 
  Lemma~\ref{lem:structure-wings}, this happens exactly if there is no
  $\mathcal{I}'$--interval nested between the $\mathcal{I}$--intervals
  defining $w_1,w_2$.  Otherwise, by Lemma~\ref{lem:structure-wings},
  there are wave-wings $w_1',w_2'$ contained in $w_1, w_2$, defined by 
  intervals $I', J' \in \mathcal{I}'$. Since we assumed
  that $I\subset J$ are directly nested, no interval in $\mathcal{I}$ nests
  between $I'$ and $J'$. Hence, by
  reversing the roles, $\mathcal{I}'(w'_1, w'_2)$  is compatible
  with $\mathcal{I}$.
\end{proof}

\begin{lemma}\label{lem:nontrivial-waves-intersected}
  Let $C$ be a cut system, and $\mathcal{I}$ an essential pattern.
  Suppose that $I\subsetneq J$ are two active intervals,
  belonging to the same curve in $C(\mathcal{I})$. 

  Let $C'$ be disjoint from $C$, and let $\mathcal{I}'$ be a pattern for $C'$ which is compatible with
  $\mathcal{I}$.  Then there is an interval $K \in \mathcal{I}'$ so that
  $I \subset K \subset J$.
\end{lemma}
\begin{proof}
  By the definition of essential pattern,
  the wings of $I \subset J$
  define nontrivial elements of $\pi_1(V)$. Hence, they cannot be
  disjoint from $C'(\mathcal{I}')$ (as, if they were, they would be disjoint
  from the filling meridian system $C'(\mathcal{I}')\cup\{\delta\}$, showing that they
  define trivial elements in $\pi_1(V)$). This is only possible if there is
  an interval $K$ as desired (again, non-nested intervals do not
  contribute intersections; compare Figure~\ref{fig:wings}).
\end{proof}
\begin{corollary}\label{cor:bound-nesting-depth}
  There is a number $D>0$ with the following property.

  Suppose that $C$ is a cut system, and $\mathcal{I}$ is an essential
  pattern for $C$.  Let $\mathcal{I}'$ be a pattern for $C'$ which is
  compatible with $\mathcal{I}$.

  Suppose that $I_1 \subsetneq I_2 \subsetneq \cdots \subsetneq I_k$
  is a chain of intervals in $\mathcal{I}$ with $k \geq D$. Then there
  is an interval $K \in \mathcal{I}'$ so that $I_1 \subset K \subset
  I_k$.
\end{corollary}
\begin{proof}
  Since a cut system has at most $g$ curves, if $D$ is large enough,
  there will be indices $i, j$ so that $I_i, I_j$ lie on the same
  curve of $C(\mathcal{I})$. Then Lemma~\ref{lem:nontrivial-waves-intersected}
  applies and yields the desired interval.
\end{proof}

\subsection{Proof of undistortion}
Before we can prove the main theorem, we need the following two results that
connect patterns to usual subsurface projections.

\begin{lemma}\label{lem:projection-compatibility}
  Suppose that $C$ is a cut system and that $C'$ is a cut system which is
  disjoint from $C$. Then for any $K>0$ there is a constant $L=L(K)>0$ so that
  the following holds.

  Let $\mathcal{I}$ be an essential pattern for $C$, and let
  $\mathcal{I}'$ be a compatible, essential pattern for $C'$.  Suppose
  that $a$ is a properly embedded arc in $\partial V - (\delta\cup C(\mathcal{I}))$.
  Assume that $i(a, C') \leq K$.  
  Then,
  \[ i(a, C'(\mathcal{I}')) < L \]
\end{lemma}
  \begin{figure}
	\centering
	\includegraphics[width=0.8\textwidth]{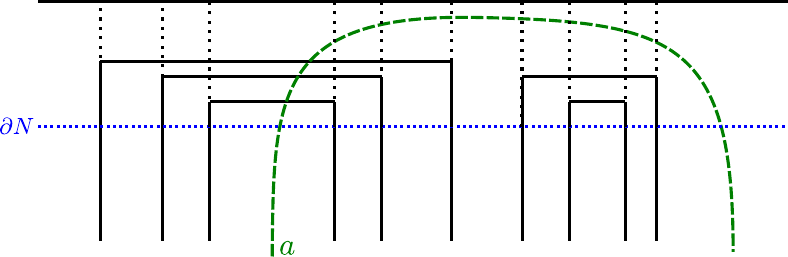}
	\caption{Our setup in the proof of Lemma~\ref{lem:projection-compatibility}. 
		Here, $C'(\mathcal{I}')$ is shown as solid lines, but $C'$ continues as vertical lines (drawn dashed) beyond the pieces used in $C'(\mathcal{I}')$. The rest of $C'(\mathcal{I}')$ consists of horizontal segments corresponding to the intervals in $\mathcal{I}'$. 
		For property (1): a subarc of $a$ which enters and leaves $N$ is forced to intersect $C'$, or is not in minimal position with respect to  $C'(\mathcal{I}')$.
	}
	\label{fig:nbigon}
\end{figure}

  \begin{figure}
    \centering
    \includegraphics[width=0.8\textwidth]{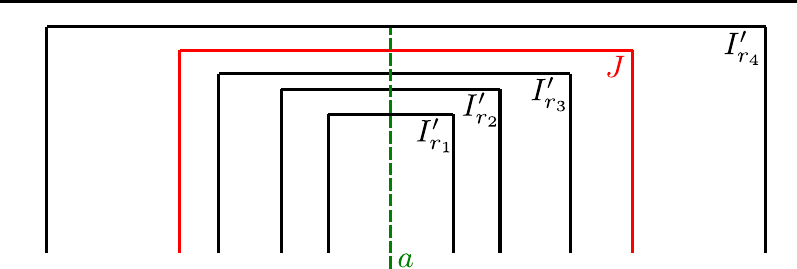}
    \caption{Property (2) in the proof of Lemma~\ref{lem:projection-compatibility}.
      As before, $C'(\mathcal{I}')$ is drawn in solid black. The arc $a$ intersects nested intervals in $\mathcal{I}'$. If the number of such intervals is too large, Corollary~\ref{cor:bound-nesting-depth} yields the existence of a nested interval $J$ of $\mathcal{I}$, which is then forced to intersect $a$.
}
    \label{fig:intersections}
  \end{figure}

\begin{proof}
  The first step of the proof is to choose suitable representatives of the isotopy classes
  of $C(\mathcal{I}), C'(\mathcal{I}')$ and $a$.

  To this end, consider an embedded collar neighbourhood 
  \[ N = (\delta^+ \cup \delta^-) \times [0,1) \] 
  of $\delta^+ \cup \delta^-$. We call the
  direction of the first coordinate \emph{horizontal}, and the other one 
  \emph{vertical}. In particular, the copies of $\delta$ forming the boundary of
  $\partial V - \delta$ are horizontal.

  Up to isotopy, we may assume that the intersections of $C, C'$ with
  the neighbourhood consist only of vertical segments.  Furthermore, we choose 
  pushoff positions of $C(\mathcal{I}), C'(\mathcal{I}')$, which are disjoint from each other and intersect
  $N$ in a union of vertical segments (which are subsegments of the
  intersections of $C, C'$ with $N$), and horizontal segments
  (corresponding to the pushed-off intervals in $\mathcal{I},
  \mathcal{I}'$). Compare Figure~\ref{fig:nbigon} for an example of such a local picture.

	\medskip
	Now consider the arc $a$, put in minimal position with respect to $C'$. We claim that (after possibly 
	changing $a$ by a further isotopy), the following two properties hold:
	\begin{enumerate}
		\item If $a'$ is a subarc of $a$ in $N$, with both endpoints on $\partial N$, then $a'$ intersects $C'$.
		\item If $a'$ is a subarc of $a$, contained in $N$ and with interior disjoint from $C'$, then $a'$ intersects
			$C'(\mathcal{I}')$ in at most $D+2$ points (where $D$ is the constant from Corollary~\ref{cor:bound-nesting-depth}.
	\end{enumerate}
	Before addressing why these claims hold, let us finish the proof of the lemma assuming the claims.
    Namely, as we assume that
	$i(a, C') < K$, we can write
	\[ a = a_1 \ast \cdots \ast a_k, \quad k \leq K \]
	into at most $K$ subsegments whose interiors are disjoint from $C'$. In particular, the endpoints
	of the $a_i$ contribute at most $K+1$ intersection points of $a$ with $C'(\mathcal{I}')$, and
	any further intersection 
	of $a$ with $C'(\mathcal{I}')$ which does not correspond to one of the endpoints of the $a_i$ lies
	on one of the horizontal segments of  $C'(\mathcal{I}')$ in $N$.
	
	By (1), each $a_i$ can be written as a concatenation of at most three subsegments, each of which is
	either completely contained in $N$, or has interior disjoint from $N$. The intersection points we have not
	yet accounted for are all contained in subsegments lying in $N$. By (2), any such subsegment intersects $C'(\mathcal{I}')$ in at most $D+2$ points. This shows that 
	$L = (K+1) + 3K(D+2) + 1$
	has the property claimed in the lemma.
	
	\medskip Finally, we explain why the properties (1) and (2) can be arranged. First, suppose that $a' \subset a$
	is a subarc in $N$ with endpoints in $\partial N$, and so that $a'$ does not intersect $C'$. Then, $a'$ bounds, 
	together with a segment in $\partial N$, a bigon $B$. Since the intersection of $C'$ with $N$ consists of
	vertical segments, and $a'$ is disjoint from $C'$, the whole bigon $B$ is disjoint from $C'$ and $C'(\mathcal{I}')$. Hence, we can perform an isotopy pushing $a$ over such a bigon, reducing the number of
	intersections with $\partial N$, and without introducing intersection with either $C'$ nor $C'(\mathcal{I}')$.
	After a finite number of such modifications, property (1) holds.
	
	Next, consider a subarc $a'$ as in (2). Any intersection of the interior of $a'$ with $C'(\mathcal{I}')$ then lies on one of the horizontal segments corresponding to an pushed-off interval in $\mathcal{I}'$. To show (2), we need
	to show that there are at most $D$ such segments. So, suppose that there are at least $D+1$ such intervals, call them $I'_1, \ldots, I'_{D+1}$. Note that since the interior of $a'$ is disjoint from $C'$, and therefore cannot cross the vertical lines bounded by the endpoints of the $I'_i$, all these intervals are nested. Hence, Corollary~\ref{cor:bound-nesting-depth} applies (with the roles of $\mathcal{I},\mathcal{I}'$ reversed), and guarantees that there is an interval $J$ of $\mathcal{I}$ nested between two of the $I'_i$. But then $a'$ intersects $C(\mathcal{I})$ in the pushed-off copy of $J$, which contradicts the assumption on $a$. Compare Figure~\ref{fig:intersections} for this scenario.
\end{proof}

%

We are now ready to prove the main theorem. The core is the following lemma.
\begin{lemma}\label{lem:generate-compatible-essential-patterns}
  Suppose that $C_n$ is a sequence of cut systems, so that consecutive $C_i$ are
  disjoint. Then there is a sequence $\mathcal{I}_n$, so that for each $n$,
  $\mathcal{I}_n$ is an essential pattern for $C_n$, and the patterns
  $\mathcal{I}_n, \mathcal{I}_{n+1}$ are compatible.
\end{lemma}
\begin{proof}
  As a first step, we apply Lemma~\ref{lem:generating-patterns}~i) and~ii) to obtain a sequence of
  surgery patterns $\mathcal{I}^1_n$ for $C_n$, so that for any $n$, the patterns $\mathcal{I}^1_n, \mathcal{I}^1_{n+1}$ are compatible. 
  The strategy of the proof will be to inductively apply Lemma~\ref{lem:adjacent-exchange} to obtain
  sequences of patterns $\mathcal{I}^k_n$ where the next sequence $(\mathcal{I}^{k+1}_n)$ is obtained from the previous one $(\mathcal{I}^k_n)$ by applying a wave exchange to a single pattern $\mathcal{I}^k_n$. For
  brevity we will simply say that the new sequence is obtained from the old one by a single wave exchange.
  
  To begin, suppose that $\mathcal{I}^1_1$ admits a wave exchange, say with wings $w_1, w_2$.
  As $\mathcal{I}^1_2$ is compatible with $\mathcal{I}^1_1$, we can apply Lemma~\ref{lem:adjacent-exchange}.
  There are two possible cases:
  First, if $\mathcal{I}^1_1(w_1, w_2)$ is compatible with $\mathcal{I}^1_2$ we can set
  $\mathcal{I}^2_1 = \mathcal{I}^1_1(w_1, w_2)$ and $\mathcal{I}^2_i = \mathcal{I}^2_i$ for $i>1$. This new
  sequence is obtained from the old one by a single wave exchange as desired.
  
  The second case is that there is a wave exchange $\mathcal{I}^1_2(w'_1, w'_2)$ which is compatible with 
  $\mathcal{I}^1_1$. In this case, we now need to apply Lemma~\ref{lem:adjacent-exchange} again, this time
  for the exchange $\mathcal{I}^1_2(w'_1, w'_2)$ and the pattern $\mathcal{I}^1_3$. If the first case happens
  here, we can define $\mathcal{I}^2_2 = \mathcal{I}^1_2(w'_1, w'_2)$ and $\mathcal{I}^2_i = \mathcal{I}^2_i$ for $i\neq 2$, as the exchange $\mathcal{I}^1_2(w'_1, w'_2)$ is now compatible with both its neighbours $\mathcal{I}^1_1, \mathcal{I}^1_3$. 
  
  Otherwise, we need to restart the process with a wave exchange of $\mathcal{I}^1_3$. Observe that
  since the sequence $\mathcal{I}^1_n$ has finite length, this process terminates either with the first case of 
  Lemma~\ref{lem:adjacent-exchange} at some index $i$, or by performing a wave exchange on the last element of the 
  sequence.
  
  In both cases, we will have constructed $\mathcal{I}^2_n$, which is obtained from $\mathcal{I}_n^1$ by a single wave exchange. If
  $\mathcal{I}^2_1$ still admits a wave exchange, we restart the process at the beginning, defining
  $\mathcal{I}^3_n$. After finitely many iterations, we obtain a sequence of patterns $\mathcal{I}^{m_1}_n$
  so that $\mathcal{I}^{m_1}_1$ does not admit wave exchanges. Note that in this case, Lemma~\ref{lem:adjacent-exchange} implies that \emph{any} wave exchange of $\mathcal{I}^{m_1}_2$ is compatible
  with $\mathcal{I}^{m_1}_1$ (as otherwise, the latter would admit a wave exchange).
  
  Thus, we can restart the procedure, trying to perform wave exchanges at $\mathcal{I}^{m_1}_2$. Arguing as above,
  after finitely many steps we will arrive at a sequence $\mathcal{I}^{m_2}_n$ where both $\mathcal{I}^{m_2}_1, \mathcal{I}^{m_2}_2$ do not admit wave exchanges. After finitely many steps, this yields desired sequence.
\end{proof}

We are almost ready to start the proof. The last missing ingredient is the following
\begin{lemma}\label{lem:cutsystem-sequence}
	Suppose that $(D,l), (D', l') \in \mathcal{G}(\delta)$ are two vertices.
	Then there is a geodesic $(C_n, l_n), n=1, \ldots, N$ in
	$\mathcal{G}$ connecting $(D,l)$ to $(D', l')$, and so that for each $1<i<N$ the 
	meridian system $C_i$ is a cut system.
\end{lemma}
\begin{proof}
	Begin by taking any geodesic $(C_n, l_n)$ joining $(D,l)$ to $(D', l')$. Let $1<i<N$ be the first
	index so that $C_i$ is not a cut system. By definition, $C_i$ is a filling meridian system, and
	therefore contains a cut system $C'_i \subset C_i$. Observe that $(C'_i, l_i)$ is a vertex of
	$\mathcal{G}$, which is still connected to both $(C_{i-1},l_{i-1})$ and $(C_{i+1},l_{i+1})$. 
	Replacing $(C_i,l_i)$ by $(C'_i, l_i)$ yields a geodesic with fewer vertices whose filling meridian
	system is not a cut system. By induction, the lemma follows.
\end{proof}

Now let $(D,l), (D',l')$ be arbitrary vertices of $\mathcal{G}(\delta)$. Up to replacing the endpoints
by vertices of distance $1$ (as in the proof of Lemma~\ref{lem:cutsystem-sequence}), we may assume that
$D, D'$ are the unions of $\delta$ with cut systems $C, C'$. Take a geodesic $(C_i,l_i)$ in $\mathcal{G}$ joining $(D,l)$ to $(D',l')$ as guaranteed by Lemma~\ref{lem:cutsystem-sequence}.

Apply Lemma~\ref{lem:generate-compatible-essential-patterns} to $C = C_1, C_2, \ldots, C_{N-1}, C_N = C'$ obtain a sequence $\mathcal{I}_n$ of patterns as in that lemma. 

Observe that for any $1\leq n\leq N$, the system $C_n(\mathcal{I}_n)$
is a meridian system which is disjoint from $\delta$. We let $\hat{C}_n$ be the filling
meridian system obtained from $C_n(\mathcal{I}_n)$ by adding every component of $\delta$
which is not already homotopic to a component of $C_n(\mathcal{I}_n)$. Observe that since the 
empty pattern is the only pattern for a cut system disjoint from $\delta$, we have $C_1(\mathcal{I}_1) = C_1 = C, 
C_N(\mathcal{I}_N) = C_N = C'$ and therefore 
$\hat{C}_1 = D, \hat{C}_N = D'$.

For each $n$, consider now
\[ A'_n = \left(\partial V - \hat{C}_n\right) \cap l_n, \]
and choose for each $n$, a subset $A_n\subset A_n'$ containing a representative of each
homotopy class of arc in $A_n'$. Define
\[ \Gamma_n = \hat{C}_n \cup A_n, \] and note that each
$\Gamma_n$ defines a vertex of $\mathcal{R}(\delta)$ as $\hat{C}_n$ contains $\delta$. Observe that
each arc $a\in A_n$ satisfies the prerequisites of
Lemma~\ref{lem:projection-compatibility} (for $C=C_n, C'=C_{n+1}, \mathcal{I}=\mathcal{I}_n, \mathcal{I}'=\mathcal{I}_{n+1}$ and $K$ depending on the constant used in the definition
of $\mathcal{G}$), and therefore the intersection number of $a$ with $\hat{C}_{n+1}$ can be bounded
in terms of the $L$ given by that lemma. Since the number of arcs in $A_n$ is bounded by
the topology of the surface, and $C_n(\mathcal{I}_n), C_{n+1}(\mathcal{I}_{n+1})$ are disjoint, 
this implies that there is a number $M>0$ so that $i(\Gamma_n, \Gamma_{n+1})<M$ for all $n$.
Hence, $\Gamma_n$ defines a path in $\mathcal{R}(\delta)$ of length
bounded above linearly by $N$.

Furthermore, this path connects the image of $(C_1,
l_1)$ and $(C_N, l_N)$ under the map
$U:\mathcal{G}(\delta)\to\mathcal{R}(\delta)$.  Hence, the distance between
$(C_1, l_1)$ and $(C_N, l_N)$ is bounded above by a linear function in $N=d_{\mathcal{G}}((D,l), (D',l'))$, showing
Theorem~\ref{thm:undistorted-stabilisers}.\qed

\begin{remark}\label{rem:compression-adapt}
  In order to prove the analogue of
  Theorem~\ref{thm:undistorted-stabilisers} for compression body
  groups, only two minor modifications are necessary. First, in
  Lemma~\ref{lem:generating-patterns}, we have to do full surgeries,
  in order to keep the systems filling meridian systems for the
  compression body. This is possible by Lemma~\ref{lem:surgery-basics}. Second, the model
  $\mathcal{G}$ needs to be adapted to compression bodies so that the loop
  $l$ is filling together with the filling meridian system (the model $\mathcal{R}$
  need not be modified).
\end{remark}

\section{Primitive and Annulus Stabilisers}
In this section we encounter distorted stabilisers in the handlebody group. There
will be two classes of such stabilisers that we consider -- those of primitive 
curves, and those of primitive annuli.

Recall that a curve $\alpha$ on the boundary of a handlebody $V$ is called
\emph{primitive} if it defines a primitive element in the (free) fundamental
group $\pi_1(V)$. Equivalently, $\alpha$ is primitive if there is a meridian
$\delta$ which intersects $\alpha$ in a single point.

A \emph{primitive annulus} will mean a pair $\alpha_1, \alpha_2$ both of which
are primitive, and which bound a properly embedded annulus $A$ in $V$.

Before beginning the discussion in earnest, we will first summarise the results
that are proven in this section. In the statements of these results we assume
that the stabilisers in question are finitely generated; this is well-known
to experts, but we will give a quick proof below.

\begin{theorem}\label{thm:primstab-distorted}
  Let $V_g, g\geq 3$ be a handlebody of genus at least three, and let
  $\alpha\subset V_g$ be primitive. Then the stabiliser of $\alpha$
  is exponentially distorted.
\end{theorem}
The genus requirement in Theorem~\ref{thm:primstab-distorted} is necessary, as the
following proposition shows.
\begin{proposition}\label{prop:primstab-genus-2}
	Let $V_2$ be a handlebody of genus $2$, and let $\alpha$ be primitive.
	Then the stabiliser of $\alpha$ is undistorted.
\end{proposition}
In fact, there are more distorted stabilisers in the handlebody group:
\begin{theorem}\label{thm:annulus-distorted}
  Let $V_g$ be a handlebody, and suppose
  that $A$ is a primitive annulus. Suppose that $g\geq 3$ if $A$ is non-separating, or that $g\geq 4$ if $A$ is
  separating. Then the stabiliser of $A$ is exponentially distorted.
\end{theorem}
\begin{remark}
  It is not clear if the genus bound in
  Theorem~\ref{thm:annulus-distorted} is optimal.
\end{remark}
To the knowledge of the author, the analogous statement of Theorem~\ref{thm:annulus-distorted} for
$\mathrm{Out}(F_n)$ is new as well:
\begin{proposition}\label{prop:cyclic-splitting-distorted}
  The stabiliser of a primitive cyclic splitting in $\mathrm{Out}(F_n)$ is exponentially distorted.
\end{proposition}

\subsection{Algebraic description of primitive stabilisers}
Stabilisers of primitive curves, in contrast to the situation of
meridians, cannot easily be reduced to lower-genus handlebody groups
and point-pushing. In this subsection we discuss some of the
difficulties encountered when trying to extend the usual description
of stabilisers using boundary pushing and a reduction to smaller genus
as for the mapping class group of a surface. A reader interested only
in the geometry of stabilisers may safely skip to the next subsection.

\bigskip Throughout, $\alpha$ will be a primitive loop on $\partial
V$. In particular, $\alpha$ is non-separating.
Let
\[ \partial V - \alpha = Y, \] where $Y$ has two boundary components
$\alpha_+, \alpha_-$ corresponding to the two sides of $\alpha$. 
Recall that in the mapping class group of $\partial V$, we have a short exact sequence
\begin{equation}
1 \to \ZZ \to \Mcg(Y) \to \mathrm{Stab}_{\Mcg}(\alpha) \to 1\label{eq:nonsep-split}
\end{equation}
where the first map sends $1$ to $T_{\alpha_+}T^{-1}_{\alpha_-}$. 
Let $\hat{Y}$ be the surface gluing a disc to 
the boundary component $\alpha_-$ of $Y$. We also have a Birman exact sequence
\begin{equation}
1 \to \pi_1(U\hat{Y}) \to \Mcg(Y) \to \Mcg(\hat{Y}) \to 1,\label{eq:birman-nonsep}
\end{equation}
where $U\hat{Y}$ denotes the unit tangent bundle of $\hat{Y}$. We call the elements
of the kernel \emph{(boundary) pushes}.
This sequence splits, for example in the following way: we define a \emph{$\alpha$--splitting
    surface} to be a subsurface $S$ of $Y$ so that $Y-S$ is a
  $3$--holed sphere containing $\alpha_-,\alpha_+$ in its
  boundary. Then the inclusion $\Mcg(S) \to \Mcg(Y)$ yields the desired splitting.
Hence, we can identify $\Mcg(Y) \cong \pi_1(U\hat{Y}) \rtimes \Mcg(\hat{Y})$.

Since $\alpha$ is primitive, neither $Y$ nor $\hat{Y}$ can be
naturally identified with a sub-handlebody. However, we can choose a
$\alpha$--splitting surface $S$ which is the boundary of a sub-handlebody,
and use the splitting of the sequence~(\ref{eq:birman-nonsep}) to
identify with quotient with $\Mcg(S)$. We then obtain an exact sequence
\[ 1 \to \pi_1(U\hat{Y})\cap\mathcal{H}(V) \to \Mcg(Y) \cap
\mathcal{H}(V) \to \mathcal{H}(S) \to 1. \]
Recall that $\mathcal{H}(S)$ is defined as the intersection of $\mathcal{H}(V)$ with $\Mcg(S)$.

Therefore, describing the stabiliser of $\alpha$ relies on describing the subgroup $\Gamma =
\pi_1(U\hat{Y})\cap\mathcal{H}(V)$ of the boundary pushing subgroup. 

Recall that we have a short exact sequence
\[ 1\to\ZZ\to \pi_1(U\hat{Y}) \to \pi_1(S) \to 1, \]
and under the identification of $\pi_1(U\hat{Y})$ with a subgroup of $\mathrm{Mcg}(Y)$,
the kernel corresponds to the Dehn twist about $\alpha$.  In
particular, since the twist about $\alpha$ is not contained in the
handlebody group of $V$, the group $\Gamma$ intersects each fibre of $\pi_1(U\hat{Y})
\to \pi_1(S)$ in at most one point.

Intuitively, it is clear that the group $\Gamma$ is much smaller than
$\pi_1(S)$. Namely, consider a meridian $\delta$ which intersects
$\alpha$ in a single point. If $a\subset Y$ is an arc based at
$\alpha_-$ disjoint from $\delta$ except in its endpoints, then the
push about $a$ maps $\delta$ to the curve obtained by concatenating
$\delta\cap Y$ with $a$. Hence, in order for the push to be in $\mathcal{H}(V)$, 
the arc $a$ would have to define a meridian as well. In fact, as the
following lemma shows, $\Gamma$ can be generated by such elements.
\begin{lemma}\label{lem:generating-pushes}
	The intersection of $\pi_1(U\hat{Y})$ with the handlebody group is generated by
	the image of all loops in $S-\alpha$ which are embedded meridians. These
    elements correspond to \emph{annular twists} $T_{\alpha}T_{\beta}^{-1}$ where 
    $\alpha, \beta$ are the boundary of a properly embedded annulus in $V$, 
    composed with Dehn twists about meridians.
\end{lemma}
Before proving Lemma~\ref{lem:generating-pushes}, we want to mention
that although pushes about embedded meridians generate $\Gamma$, the
group does not simply consist of pushes along $V$--trivial arcs. In fact, we
have the following.
\begin{lemma}\label{lem:stablisers-non-normal}
	For $V$ of genus $g\geq 3$, the group $\Gamma$ is not normal in $\pi_1(U\hat{Y})$.
\end{lemma}
We prove Lemma~\ref{lem:stablisers-non-normal} in the appendix, since the proof
only consists of a careful, somewhat lengthy check of intersection patterns.
However, we want to emphasise the following consequence of
Lemma~\ref{lem:stablisers-non-normal} in combination with
Lemma~\ref{lem:generating-pushes}, which may be of independent
interest, and highlights another difference between the complements of meridians
and primitives in a handlebody.
\begin{corollary}
  The kernel $\ker(\pi_1(S-\alpha) \to \pi_1(V))$ of the map induced
  by inclusion is not generated by embedded curves.
\end{corollary}

\begin{proof}[Proof of Lemma~\ref{lem:generating-pushes}]
  Recall that $\Gamma = \pi_1(U\hat{Y}) \cap \mathcal{H}$, and denote by $\Gamma_0$ 
  the subgroup generated by the pushes as in the statement of the lemma. Pick a point $p \in \alpha$,
  and note that it defines points $p_-, p_+ \in Y$. Define
  $\mathcal{A}$ to be the graph whose vertices correspond to homotopy classes of arcs $a\subset Y$ joining $p_-$
  to $p_+$, so that $a$ defines a meridian on $\partial V$. We join two vertices
  with an edge, if the corresponding arcs are disjoint except at their endpoints. 
  Note that $\mathcal{H}(V)\cap\Mcg(Y)$ acts on $\mathcal{A}$ as isometries.

  Observe that there is an arc $a$ as above, so that $Y\setminus a$ is homotopy
  equivalent to the splitting surface $S$. Hence, the stabiliser of this arc $a$ in 
  $\mathcal{H}(V)\cap\Mcg(Y)$ is equal to (the image of) $\mathcal{H}(S)$.

  To prove the lemma, it therefore suffices to show that $\Gamma_0$ acts transitively on the
  vertex set of $\mathcal{A}$. To this end, first consider two arcs $a, a'$ corresponding to adjacent
  vertices of $\mathcal{A}$.
  Then the concatenation $l=a^{-1}*a'$ is a loop joining $p_-$ to itself, and furthermore it defines
  an embedded meridian in $Y$. Hence, the push $P(l)$ about $l$ is an element of the handlebody group
  (it is an annular twist). Furthermore, we have
  \[ P(l)(a') = a \] 
  Hence, to prove the claim, it suffices to show
  that $\mathcal{A}$ is connected.  This follows from a standard
  surgery argument: suppose $a, a'$ are any two arcs representing
  vertices that are not disjoint. Since they both define meridians on
  $\partial V$, there is a wave $w\subset a'$.  A suitable surgery
  $a_w$ then intersects $Y$ in an arc still connecting $p_-$ to $p_+$,
  which is otherwise disjoint from $a$, and intersects $a'$ in strictly fewer points. 
\end{proof}

\subsection{Upper distortion bounds}
The upper distortion bounds in Theorem~\ref{thm:primstab-distorted},~\ref{thm:annulus-distorted} and Proposition~\ref{prop:cyclic-splitting-distorted} follow from a surgery construction. We begin by describing the
case of an annulus $A$ in a handlebody in detail; the case of a primitive element is very similar.
Then we discuss the case of a primitive cyclic splitting in the free group.

\medskip
We are again using the two complexes $\mathcal{G}$ and $\mathcal{R}$
which appeared in Section~\ref{sec:good-stabs}. In fact, we consider
the following sub-complex:
\begin{definition}
	Let $\mathcal{G}(A)$ to be the full sub-complex of $\mathcal{G}$ of
	all those vertices whose cut system $C$ intersects each curve in
	$\partial A$ in exactly one point, and also so that $l$ intersects
	each curve in $A$ at most in one point.
\end{definition}
\begin{definition}
	Let $\mathcal{R}(A)$ be the full sub-complex of $\mathcal{G}$ of
	all those vertices whose cut system $C$ intersects each curve in
	$\partial A$ in exactly one point, and also so that $\partial A$
	embeds in the graph as an embedded subgraph.
\end{definition}
\begin{lemma}
	Suppose that $C, C'$ are cut systems both of which intersect each component of $A$
	in exactly one point. If $C$ and $C'$ are not disjoint, then there is a surgery $C_1$ 
	of $C$ in direction of $C'$ (i.e. defined by a wave of a component $c' \in C'$ with respect to $C$)
	which also intersects each component of $A$ in exactly one point. In addition, we may assume that
	this wave is disjoint from $A$.
\end{lemma}
\begin{proof}
	Let $\alpha$ be one of the boundary components of $A$. Since $C'$ intersects $C$,
	there is are two distinct waves $w_1, w_2$ of some component of $C'$ with respect
	to $C$. 
	
	Since $\alpha$ intersects $C'$ in a single point, we may assume without loss of 
	generality that $w_1$ does not intersect $\alpha$. 
	Let $C_1$ be the cut system surgery defined
	by that wave $w_1$. As the wave $w_1$ is disjoint from $\alpha$, the
	result $C_1$ intersects $\alpha$ in at most one point. As $\alpha$
	is nontrivial in $\pi_1(V)$, it cannot be disjoint from a cut system
	-- hence, $\alpha$ intersects $C_1$ in a single point.
	
	Consider now the second boundary component $\beta$ of $A$. Since
	$\beta$ intersects both $C$ and $C'$ in one point, it intersects
	$C_1$ in at most two points. Let $D_1$ be a collection of disjoint,
	properly embedded discs, bounded by $C_1$. Then $\alpha$ intersects
	$D_1$ in a single point. As $\alpha, \beta$ are freely homotopic,
	$\beta$ also intersects $D_1$ in an odd number of points (as the
	parity of intersection can be detected by the algebraic intersection
	pairing, and therefore it is an invariant of free homotopy classes).
	Hence, it is impossible that $\beta$ intersects $C_1$ in zero or two 
	points, and $\beta$ intersects $C_1$ in a single point as claimed.
\end{proof}
As a consequence, we get the following two results
\begin{corollary}\label{cor:surgery-length-bound}
	Given two cut systems $C, C'$, both of which intersect each component of $A$
	in a single point, there is a sequence
	\[ C = C_0, C_1, \ldots, C_n, C_{n+1} = C'\] of cut systems so that all $C_i$
	intersect each component of $A$ in a single point, $n \leq i(C, C')$, and consecutive $C_i, C_{i+1}$ are 
	obtained by surgery at a wave $w_i$ disjoint from $A$.
\end{corollary}
\begin{corollary}\label{cor:finite-generation}
	After possibly increasing the constants chosen in their definitions, the graphs $\mathcal{G}(A), \mathcal{R}(A)$ are connected. The stabiliser of $A$ acts on them
	properly discontinuously and cocompactly. Hence, this stabiliser is in particular finitely generated.
\end{corollary}
\begin{proof}
	By choosing the constants in the fefinition of $\mathcal{R}$ large enough, any two vertices whose underlying filling meridian system
	is the same can be joined by a path with the same meridian system (compare the proof of Lemma~\ref{lem:modelbuilding}). Hence, the same is 
	true in $\mathcal{R}(A)$ (as one can simply keep the edges corresponding to $A$).
	Now, Corollary~\ref{cor:surgery-length-bound} implies connectivity by arguing as in
	the proof of Lemma~\ref{lem:modelbuilding}.
	The fact that the stabiliser of $A$ acts cocompactly is also shown exactly as in
	Lemma~\ref{lem:modelbuilding}.
\end{proof}
We also note the following (compare e.g. \cite[Corollary~A.4]{HH1} for a similar result)
\begin{proposition}\label{prop:intersection-growth}
	There are numbers $a,b$ so that the following holds. Let $(C,l)\in\mathcal{G}$ be 
	any vertex, and $f\in\mathcal{H}(V)$ be a handlebody group element.
	Then
	\[ i(C\cup l, f(C\cup l)) \leq a\cdot b^{\|f\|_{\mathcal{H}(V)}}. \]
\end{proposition}
\begin{proof}
	Since $\mathcal{G}$ and $\mathcal{H}(V)$ are quasi-isometric, it suffices to show that
	$i(C\cup l, C'\cup l')$ can be bounded by an exponential function of $d((C,l),(C',l'))$.
	To show this, it in turn suffices to show that there is a constant $K$ so that
	\[ i(C\cup l, C'\cup l') \leq Ki(\hat{C}\cup \hat{l}, C'\cup l') \]
	whenever $(\hat{C}, \hat{l}), (C,l)$ are adjacent in $\mathcal{G}$. Denote by $D_1, \ldots, D_k$ the
	complementary components of $C\cup l$. Observe that $k$ can be bounded in terms of $i(C,l)$, which 
	in turn is bounded by a constant chosen in the definition of $\mathcal{G}$. Denote by $c'_1, \ldots, c'_N, l'_1, \ldots, l'_M$ the intersection arcs of $C', l'$ with the $D_j$. The total number $N+M$ of these arcs is bounded
	by $i(C\cup l, C'\cup l')$. 
	
	Now, since $(C,l), (\hat{C}, \hat{l})$ are adjacent we can similarly write $\hat{C}$ and $\hat{l}$ as concatenations of arcs $\hat{c}_1, \ldots, \hat{c}_r, \hat{l}_1, \ldots, \hat{l}_s$ in the $D_j$.
	Here, the total number $r+s$ is again uniformly bounded, dependent only on the constants defining $\mathcal{G}$.
	Since all $D_i$ are discs, up to homotopy fixing endpoints, any $\hat{c}_i$ or $\hat{l}_i$ and any $c'_j$ or $l'_j$ intersect in at most one point.
	Hence, we have
	\[ i(\hat{C}\cup\hat{l}, C'\cup l') \leq (r+s)(N+M) \leq Ki(C\cup l, C'\cup l'),\]
	for a suitable choice of $K$. This shows the proposition.
\end{proof}

The key to the upper distortion bound lies in the following two lemmas,
which we will use to inductively build a path.

For their formulation, suppose that $\Gamma$ is a graph representing a vertex of $\mathcal{R}$.
Recall that this means that, in particular, there is an embedded cut system $C \subset \Gamma$. 
We denote this system by $C(\Gamma)$, and we call any edge of $\Gamma$ which is not contained
in $C$ a \emph{rope edge}.

\begin{lemma}\label{lem:repair-ropes}
	Let $C$ be a cut system, so that $C$ intersects each component of $A$ in a single point.
	Suppose
	that $\Gamma_i$ is a vertex of $\mathcal{R}(A)$, so that each rope edge
	$e$ of $\Gamma_i$ intersects $C$ in at most $K$ points.
	Then there is a vertex $\Gamma_i'$
	of $\mathcal{R}(A)$ with the following properties:
	\begin{enumerate}[i)]
		\item $C(\Gamma_i) = C(\Gamma'_i)$.

		\item Each rope edge of $\Gamma'_i$ intersects $C$ in at most $K$ points.
		\item Each arc in $C\cap (\partial V - C(\Gamma'_i))$ which is disjoint from $A$, is disjoint
		from the rope edges of $\Gamma_i'$ up to homotopy.
		

		
		\item The distance between $\Gamma_i$ and $\Gamma'_i$ in $\mathcal{R}(A)$
		is at most $L_1K$. Here, $L_1$ is a constant depending only on the topological
		type of $S$, and the constants defining the graph $\mathcal{R}$.
	\end{enumerate}
\end{lemma}
\begin{proof}
	To obtain $\Gamma_i'$ from $\Gamma_i$, we will successively replace rope edges using surgery
	using segments in $C$.
	None of these moves change the underlying meridian system, guaranteeing property~i), and cannot
	introduce new intersections with $C$, guaranteeing property~ii).
	
	For any pair of a $C$--arc disjoint from $A$, and rope edge, at
	most $K$ surgeries are needed to make them disjoint (by assumption on the intersection number
	of rope edges and $C$--arcs), and each surgery step stays in
	$\mathcal{R}(A)$. Since the
	number of different $C$--arcs is uniformly bounded by the genus of
	$\partial V$, and the same is true for the number of rope edges of $\Gamma_i$,
	this shows that after at most $L_1K$ steps we arrive at the desired $\Gamma'_i$
	(where $L_1$ just depends on the number of possible topological types of rope edges and $C$--arcs).

\end{proof}
\begin{lemma}\label{lem:surgery-step}
	Suppose that $\Gamma'_i$ is a vertex of $\mathcal{R}(A)$, and
	 that $C_{i+1}$ is a cut system obtained from $C_i = C(\Gamma'_i)$ from
	a surgery move in the direction of $C$, defined by a wave $w$ disjoint from $A$.
	
	Further suppose that each rope edge of $\Gamma'_i$ is disjoint from $w$, and that
	each rope edge of $\Gamma'_i$ intersects $C$ in at most $K$ points.
	
	\smallskip Then there is a vertex $\Gamma_{i+1}$ of $\mathcal{R}(A)$ with the following properties:
	\begin{enumerate}[i)]
		\item $C(\Gamma_{i+1}) = C_{i+1}$.
		\item Every rope edge of $\Gamma_{i+1}$ intersects $C$ in at most $\max\{K, i(C, C_i)\}$ points.
		\item The distance between $\Gamma'_i, \Gamma_{i+1}$ in $\mathcal{R}(A)$ is at most $1$.
	\end{enumerate}
\end{lemma}
\begin{proof}
	First, by possibly adding $w$ as an additional rope edge, we may
	replace $\Gamma'_i$
	by an adjacent vertex $\Gamma''_i$ so that $w$ is a rope edge. 

	Then, $\Gamma''_i$ contains $C_{i+1}$ as a subgraph, and we
	define $\Gamma_{i+1}$ to be this vertex, guaranteeing i). In particular, observe that $w$
	has ceased to be a rope edge of $\Gamma_{i+1}$ (as it is now part of the meridian system),
	and $\Gamma_{i+1}$ instead has a rope edge $c$ which is a subarc of $C_i$. In particular,
	that rope edge intersects $C$ in at most $i(C, C_i)$ points. 	
	Any other rope edge of $\Gamma_{i+1}$ is also a rope edge of $\Gamma'_i$, so by assumption it  
	has at most $K$ intersections with $C$. Hence, property~ii) holds.
	
	Property~iii) is clear from the construction.
\end{proof}

\begin{proof}[Proof of the upper bound in Theorem~\ref{thm:annulus-distorted}]
	Fix a basepoint $(C,l)\in\mathcal{G}(A)$, so that the two $A$--arcs can be made disjoint from
	the $l$--arcs by homotopy. This exists, assuming that the constants defining $\mathcal{G}(A)$ are
	chosen large enough: e.g. by starting with the $A$--arcs,
	and adding enough additional arcs so that any boundary of $S-C$ is joined to any other. The resulting $l$
	intersects any curve in $C$, and every wave relative to $C$, and therefore is discbusting. The number of arcs
	necessary depends only on the number of boundary components of $S-C$ (hence, the genus of $S$), and thus if
	the constant defining $\mathcal{G}$ is chosen large enough, $(S,l)$ defines a vertex.
	
	\smallskip Now consider an element $f\in\mathrm{Stab}(A)$. We then
	know, from Proposition~\ref{prop:intersection-growth} that
	\[ i(C\cup l, f(C\cup l)) \leq a\cdot b^{\|f\|_{\mathcal{H}(V)}} = M. \]
	Let $\Gamma \subset C \cup l$ be the vertex of $\mathcal{R}$
	corresponding to $(C, l)$, and let $\Gamma_0 = f(\Gamma)$ be the image under $f$. 
	
	We begin by applying Corollary~\ref{cor:surgery-length-bound}
	to obtain a surgery sequence $C_1,\ldots, C_L$ from $f(C)$ to $C$, of length $L \leq M$. We have $i(C_j, C) \leq M$ for all $j$, and each surgery step in this sequence is done
	using a wave disjoint from $A$.
	
	Next, we apply Lemma~\ref{lem:repair-ropes} to $\Gamma_0$ and $M$ as the intersection bound.
	We obtain a vertex $\Gamma'_1$, with $d(\Gamma_0, \Gamma'_1) \leq L_1M$, and so that $\Gamma'_1$
	satsifies the prerequisites of Lemma~\ref{lem:surgery-step}. Applying the latter lemma yields
	a vertex $\Gamma_2$ of distance $\leq 1$, with $C(\Gamma_2) = C_2$, and where each rope still 
	has intersection at most $\max\{M, i(C_1, C)\}\leq M$ with $C$.
	
	Thus, we can inductively build a path of length $n(L_1M + 1)\leq M(L_1M+1)$ joining $f(\Gamma)$
	to a vertex $\Gamma_n$ with $C(\Gamma_n) = C$.
	Recall that the stabiliser of $C$ in the handlebody group of $V$ is equal to the stabiliser of $C$
	in the mapping class group of $\partial V$. Hence, the stabiliser of $C\cup A$ in the handlebody group
	is also equal to the stabiliser of $C\cup A$ in the mapping class group. In the mapping class group,
	stabilisers of curve systems, or arc systems, are undistorted \cite{MM2}. Hence,
	the stabiliser of $C\cup A$ is also undistorted in the handlebody group, and therefore there is a path
	of length coarsely bounded by $M(L_1M+L_2)$ joining $\Gamma_n$ to $\Gamma$ in $\mathcal{R}(A)$.
	
	As these length bounds are polynomial in $M$, the triangle inequality yields that the distance
	between $\Gamma$ and $f(\Gamma)$ in $\mathcal{R}(A)$ can be bounded by an exponential
	in $\|f\|_{\mathcal{H}(V)}$, showing that $\mathcal{R}(A)$ is at most exponentially distorted.
	This shows the upper bound in Theorem~\ref{thm:annulus-distorted}.
\end{proof}

To prove Proposition~\ref{prop:cyclic-splitting-distorted}, we
work in a doubled handlebody.  First, using surgeries of sphere
systems instead of meridian systems we show the following analogue of
Corollary~\ref{cor:surgery-length-bound} with essentially the same argument.
\begin{lemma}\label{lem:sphere-surgery-length-bound}
	Let $W = \#_gS^1\times S^2$ be a doubled handlebody, and suppose
	that $T\subset W$ is an embedded torus so that the image of
	$\pi_1(T)\to\pi_1(W)$ is a cyclic group generated by an primitive 
	element of the free group $\pi_1(W)$.
	
	Suppose that $\sigma, \sigma'$ are two sphere systems in minimal
	position, each of which intersects $T$ in a single circle. Then
	there is a sequence
	\[ \sigma = \sigma_0, \sigma_1, \ldots, \sigma_n, \sigma_{n+1} = \sigma'\] 
	so that each $\sigma_i$ intersects $T$ in a single circle, and $n$ 
	is at most the number of intersection circles in $\sigma\cap\sigma'$.
\end{lemma}
\begin{proof}
	Consider $\sigma\cap \sigma'$, which is a collection of disjoint circles. Consider
	any innermost circle $C$ of this collection, i.e. a circle which bounds a disc
	$D\subset \sigma'$ whose interior is disjoint from $\sigma$. As there are at least
	two such innermost circles, we can choose one where $D$ is disjoint from the torus $T$.
	
	We define $\sigma_1$ as the surgery of $\sigma$ using this disc $D$. As $T$ cannot become
	disjoint from a filling sphere system (otherwise the image of $\pi_1(T)$ in $\pi_1(W)$ would
	be trivial), $T$ will intersect $\sigma_1$ in a single circle as well.
	
	Further, $\sigma_1$ has at least one fewer intersection circle with $\sigma'$, and so the
	lemma follows by induction.
\end{proof}
We also have the following analog of Proposition~\ref{prop:intersection-growth}:
\begin{lemma}
	Let $W = \#_gS^1\times S^2$ be a doubled handlebody. Then there are
	numbers $a,b$ so that the following holds.  Let $\sigma$ be any
	filling sphere system, and $f\in\mathrm{Mcg}(W)$ be arbitrary.
	Then, in minimal position, the number of intersection circles in
	$\sigma\cap f(\sigma)$ is at most 
	\[ a\cdot b^{\|f\|_{\mathrm{Mcg}(W)}}. \]
\end{lemma}
\begin{proof}
	To show the lemma, it suffices to show that there is a number $C$ so
	that if $\sigma, \sigma'$ are any two sphere systems, and $\sigma''$
	is disjoint from $\sigma'$, then
	\[ i(\sigma, \sigma'') \leq Ci(\sigma, \sigma') + C \] This follows
	since $\sigma''$ intersects each component of
	$\sigma\cap(W-\sigma')$ in at most one circle.
\end{proof}
Together with Lemma~\ref{lem:sphere-surgery-length-bound}, this lemma proves
the upper bound in Proposition~\ref{prop:cyclic-splitting-distorted} by induction.

\bigskip
Finally, we prove the undistortion statement in genus $2$ for primitive stabilisers.
\begin{proof}[Proof of Proposition~\ref{prop:primstab-genus-2}]
	As in Section~\ref{sec:good-stabs}, we aim to project paths in $\mathcal{R}$ to $\mathcal{R}(\alpha)$.
	First, we observe the following preliminary step. Suppose that $\Delta=\{\delta_1, \delta_2\}$ is
	any cut system. Then, since $\alpha$ is primitive, at least one of $\iota(\alpha, \delta_1),\iota(\alpha,\delta_2)$ is odd. In particular, there is a subarc $d \subset \delta_1\cup\delta_2$ connecting the two different sides of $\alpha$. One component of the boundary of 
	a regular neighbourhood of $d\cup\alpha$ is a separating meridian $\delta(d)$. As $V_2$ has genus $2$, there is a unique cut system
	$\Delta(d)$ disjoint from $\delta(d)$. Since $\alpha$ is disjoint from $\delta(d)$, it intersects this cut system in a single point.
	Observe
	that if $d'$ is any other possible choice of arc, the meridians $\delta(d), \delta(d')$
	intersect in at most four points, and thus $\Delta(d), \Delta(d')$ also intersect in uniformly few points. 
	
	The same argument shows that if $\Delta'$ is a disjoint cut system, and $d'$ is an admissible arc, then $\Delta(d), \Delta'(d')$ intersect in uniformly few points. Hence,
	we can define a Lipschitz projection of $\mathcal{H}_2$ to the stabiliser of $\alpha$.
\end{proof}

\subsection{Lower distortion bounds}
The proofs of the lower distortion bounds for all three results mentioned at the beginning of this section are
very similar, and rely on two main ingredients.  On the one hand, we
use the following theorem, which is shown by Handel-Mosher (\cite[Section~4.3, Case~1]{HM}):
\begin{theorem}\label{thm:hm-distortion}
  Let $n\geq 3$ be given, and $F_n$ is a free group with free basis $e_1,\ldots,e_n$. Suppose that $\Theta:\langle e_1,e_2\rangle\to \langle e_1,e_2\rangle$ is
  an irreducible automorphism of exponential growth. Define an automorphism $f_k\in\mathrm{Out}(F_n)$ by the rule
  \[ e_i \mapsto e_i, \quad i<n \]
  \[ e_n \mapsto e_n\Theta^k(e_1). \] Then the norm of $f_k$ in the stabiliser of the conjugacy class $[e_1]$ grows exponentially in
  $k$.
\end{theorem}
The second ingredient is a construction similar to the one employed in
Section~3 of \cite{HH2}.

Namely, let $X$ be a surface of genus $1$ with one boundary
component. Consider the $3$--manifold $W = X \times [0,1]$, which is a
handlebody of genus $2$. The boundary 
\[ \partial W = X^0 \cup A \cup X^1, \quad X^i = \{i\}\times X, A = \partial X \times [0,1] \] consists of two copies of $X$ and an annulus $A$. There is a map
\[ \iota:\mathrm{Mcg}(X) \to \mathcal{H}(W), \] 
which maps a homeomorphism $f$ of $X$ to the homeomorphism $f \times \mathrm{id}$ of $W$.

Choose $2(g-2)$ disjoint discs $D^-_i, D^+_i \subset \mathrm{int}(A)$, 
for $i=3,\ldots, g$ and for each $i$ attach a three-dimensional 
$1$-handle $h_i$ to $D_i^+, D_i^-$ to obtain a handlebody $V$ of genus $g$. 
Observe that homeomorphisms of $X$ which have the form $f\times\mathrm{id}$ 
restrict to the identity on the annulus $A$, and therefore the map $\iota$ yields
a map
\[ \iota':\mathrm{Mcg}(X) \to \mathcal{H}(V) \] 
by sending $f$ to the homeomorphism which restricts to $f\times\mathrm{id}$ on $W$,
and to the identity on all handles $h_i$.

\medskip Let $\alpha \subset X$ be a non-separating simple closed curve, and denote
by $\alpha^j = \alpha\times\{j\}$ for $j=0,1$. Observe that $\alpha^0, \alpha^1$ are homotopic
in $V$, and are primitive. Let $\beta^0 \subset(A\setminus \cup_i(D_i^+\cup D_i^+))\cup X^0$ be a simple closed curve which bounds a pair of pants together with $\partial D_g^-$ and
$\alpha^0$ (compare Figure~\ref{fig:pushing-setup}).  
\begin{figure}
  \centering
  \includegraphics[width=\textwidth]{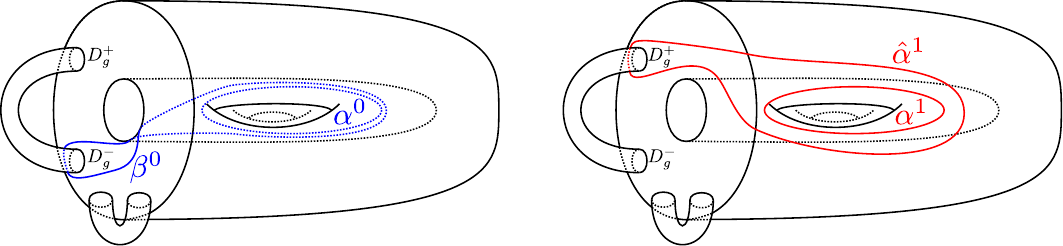}
  \caption{Left: The setup to construct distorted curve
    stabilisers. Right: A non-separating annulus fixed by the elements
    $f_k$. A separating annulus with the same property could be
    constructed by making $\hat{\alpha}^1$ surround both discs $D_g^-, D_g^+$ 
    where the lower handle is attached.}
  \label{fig:pushing-setup}
\end{figure} 
Consider the mapping class
\[ P = T_{\alpha^0}T_{\beta^0}^{-1} \in \mathcal{H}(V), \]
which is a handle slide.

Choose a basis $e_1, \ldots, e_g$ for $\pi_1(V)$, so that $e_1, e_2$
correspond to loops in $X^0$, the loop $\alpha^0$ defines the
conjugacy class of $e_1$, and the loops $e_3, \ldots, e_g$ are dual to
the handles $h_i$, not entering $X^0\cup
X^1$. 
We summarise some important properties of $P$ in the following lemma.
\begin{lemma}\label{lem:action-P}
  The element $P$ induces the following automorphism on $\pi_1(V)$ with respect to
  the basis chosen above:
  \[ e_i \mapsto e_i, i < g \]
  \[ e_g \mapsto e_g e_1. \] Furthermore, $P$ fixes the curve
  $\alpha^0$ and restricts to the identity on $X^1$.
\end{lemma}
Let $\psi$ be a pseudo-Anosov element of $X$ which induces an
irreducible, exponentially growing automorphism $\Theta$ of $\pi_1(X) =
F_2$. Recall that $\Psi = \iota'(\psi)$ is then an element of the handlebody group of
$V$, which restricts to $\psi$ on $X$. Define the elements 
\[ f_k = \Psi^kP\Psi^{-k} \]
Each $f_k$ lies in the handlebody group, and we have the following
\begin{lemma}\label{lem:action-fk}
  The element $f_k$ induces the following automorphism on $\pi_1(V)$ with respect to
  the basis chosen above:
  \[ e_i \mapsto e_i, i < g \]
  \[ e_g \mapsto e_g\Theta^k(e_1). \] 
  Furthermore, each $f_k$ fixes
  \begin{enumerate}[i)]
  \item the curve $\alpha^1$,
  \item a non-separating annulus $A$, one boundary component of which is $\alpha^1$,
  \item for each $h=1, \ldots, g-2$, a
    separating annulus $A_h$, one boundary component of which is
    $\alpha^1$, and so that one complementary component of $A_h$ has genus $h$.
  \end{enumerate}
  Here, the annuli $A, A_h$ do not depend on $k$.
\end{lemma}
\begin{proof}
  The claim on the action on fundamental group is clear from Lemma~\ref{lem:action-P}.
  Also observe that since $P$ acts as the identity on $X^1$, the same is true for $f_k$.
  This immediately implies that the mapping class $f_k$ preserves the loop $\alpha^1 \subset
  X^1$.
  Furthermore, we can choose a curve $\hat{\alpha}^1$ which is disjoint
  from $\alpha^0, \beta^0$ and bounds an annulus $A$ together with
  $\alpha^1$ (compare Figure~\ref{fig:pushing-setup}). By choosing the $\partial D_i^\pm$ to lie on the correct
  side of $\hat{\alpha}^1, \alpha^1$, we can ensure that
  the annulus $A$ can be non-separating or separating, and in the latter
  case, we can choose the genus separated off freely between $1$ and
  $g-2$. 
\end{proof}
Now we are ready to prove the lower distortion parts of the theorems
mentioned at the beginning of this section.
\begin{proof}[Proof of the lower bound in Theorem~\ref{thm:primstab-distorted}]
The stabiliser of any primitive curve $\alpha$ is conjugate, in the
handlebody group, to the stabiliser of $\alpha^1$. Hence, it suffices to show
that the stabiliser of $\alpha^1$ is at least exponentially distorted. We
use the elements $f_k$ as above. 

Suppose that $f\in\mathrm{Stab}_{\mathcal{H}(V)}(\alpha^1)$ is given. Then,
since $\alpha^1$ (and $\alpha^0$) define the conjugacy class $[e_1]$ in 
$\pi_1(V)$, the induced automorphism $f_*\in \mathrm{Out}(\pi_1(V))$ fixes
the conjugacy class $[e_1]$. In other words, there is a group homomorphism
\[ \pi:\mathrm{Stab}_{\mathcal{H}(V)}(\alpha^1) \to
\mathrm{Stab}_{\mathrm{Out}(F_g)}([e_1]). \]
Recall that for any choice of word norms, group homomorphisms are Lipschitz 
maps. By Theorem~\ref{thm:hm-distortion}, the elements $\pi(f_k)$ have
norm growing exponentially in $k$. Hence, the sequence $f_k$
has norms growing at least 
exponentially in $\mathrm{Stab}_{\mathcal{H}(V)}(\alpha^1)$. On the other hand, as $f_k =
\Psi^kP\Psi^{k}$, the norm of $f_k$ in $\mathcal{H}(V)$ is clearly
growing linearly in $k$. This shows that
$\mathrm{Stab}_{\mathcal{H}(V)}(\alpha^1)$ is at least exponentially
distorted in $\mathcal{H}(V)$.
\end{proof}
\begin{proof}[Proof of the lower bound in Theorem~\ref{thm:annulus-distorted}]
  The stabiliser of any annulus as in that theorem is conjugate to an
  annulus $A$ or $A_h$ as in Lemma~\ref{lem:action-fk}. Now, we can
  finish the proof just like the previous argument. Namely, the
  elements $f_k$ as above fix $A, A_h$, and also the stabilisers of
  these annuli are contained in the stabiliser of $\alpha^1$.
\end{proof}
As mentionend in the introduction, the lower distortion bound in
Proposition~\ref{prop:cyclic-splitting-distorted} follows directly from \cite{HM}.
For completeness, we include a proof (from a topological perspective).
\begin{proof}[Proof of the lower bound in Proposition~\ref{prop:cyclic-splitting-distorted}]
  We use  the connection of $\mathrm{Out}(F_n)$ to the mapping class group
of a the double of a handlebody. Let $W$ be the closed $3$--manifold obtained by doubling
$V$ about its boundary. Recall the short exact sequence \cite[{Th\'{e}or\`{e}me 4.3, Remarque 1)}]{L74}
\[ 1 \to K \to \mathrm{Mcg}(W) \to \mathrm{Out}(F_n) \to 1 \] where
$K$ is finite, and the right map is induced by the action on the
fundamental group.  We also have a natural map $\mathcal{H}\to
\mathrm{Mcg}(W)$ obtained by doubling, so that the composition
$\mathcal{H}\to \mathrm{Mcg}(W)\to \mathrm{Out}(F_n)$ agrees with the
action on the fundamental group of the handlebody.

Under the doubling map $\mathcal{H}\to \mathrm{Mcg}(W)$, the
stabiliser of an annulus $A$ as above in $\mathcal{H}$ maps to the
stabiliser of a torus $T_A$, so that the image of $\pi_1(T_A)$ in
$\pi_1(W)$ is generated by $[e_1]$.  Under the map $\mathrm{Mcg}(W)
\to \mathrm{Out}(F_n)$ the stabiliser of $T_A$ maps to the stabiliser
of a primitive cyclic splitting $Z$, where the amalgamating group is generated
by $[e_1]$.  The same argument as above then shows that the image of
the sequence $f_k$ has length growing exponentially in $k$ in the
stabiliser of $Z$. As all primitive cyclic splittings differ by an element of
$\mathrm{Out}(F_n)$, this shows that all stabilisers of primitive cyclic splittings
are at least exponentially distorted.
\end{proof}

\newpage
\appendix
\section{The proof of Lemma~\ref{lem:stablisers-non-normal}}
We give the proof in the case of a genus $3$ handlebody, but the method extends
to any genus $\geq 3$.
\begin{figure}
	\centering
	\includegraphics[width=0.49\textwidth]{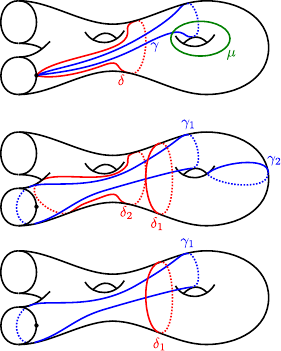}
	\includegraphics[width=0.49\textwidth]{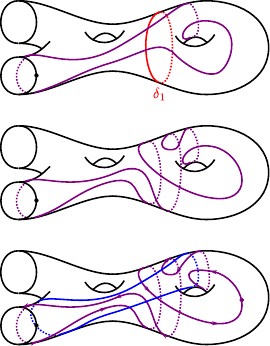}
        
        \vspace{1 em}
	\includegraphics[width=0.49\textwidth]{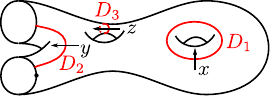}
	\caption{The left three pictures show the relevant curves in
		the proof of Lemma~\ref{lem:stablisers-non-normal}. For ease of depiction, in all of these pictures the
		handlebody structure is the ``outside'' handlebody in the standard
		Heegaard splitting of $S^3$.
		The three pictures on the right depict the action of the twist
		product on $\mu$: the top shows $T_{\gamma_1}\mu =
		T_{\gamma_1}T_{\delta_1}\mu$, the middle one shows
		$T_{\delta_1}^{-1}T_{\gamma_1}T_{\delta_1}\mu$ and in the bottom one
		$\gamma_1$ is shown superimposed.
              Below is the basis used to compute the element after applying all twists.}
	\label{fig:setup-example} 
      \end{figure} 

	Consider two disjoint loops $\gamma, \delta$ as in
	Figure~\ref{fig:setup-example} on the left, based at the curve
	$\alpha_-$, i.e. loops so that under the identification of $\hat{Y}$ with the
	splitting surface $S$, the loop $\delta$ is a separating meridian, while
	$\gamma$ is a primitive curve.  We will show that the boundary push
	about any element in the fibre of $\pi_1(U\hat{Y}) \to \pi_1(S)$
	over the commutator $[\delta,\gamma]$ is not contained in the handlebody
	group. This is enough to prove the lemma.
	
	
	\smallskip
	A push along $\delta$ will be of the form
	\[ P_\delta = T^{-1}_{\delta_1}T_{\delta_2}T_\alpha^k \] 
	for some $k$, where $\delta_1,
	\delta_2$ are disjoint from $\delta$ and bound a pair of pants together with $\alpha$ (see Figure~\ref{fig:setup-example} for this setup). Here and below, bounding
	a pair of pants, and twists are meant as objects on $S$, before cutting at $\alpha$. 
	Similarly, a push along
	$\gamma$ will be of the form 
	\[T^{-1}_{\gamma_1}T_{\gamma_2}T_\alpha^l\]
	for some $l$, and $\gamma_1, \gamma_2$ disjoint and bounding a pair of pants with $\alpha$. 
	
	Since $\delta_2, \gamma_2$ are disjoint from all
	other involved curves, and the corresponding Dehn twists therefore commute with all others
	involved in the definition of $P_\delta, P_\gamma$, we can compute the commutator of the pushes as
	\[ \Psi = [P_\gamma, P_\delta] = T_{\gamma_1}^{-1} T_{\delta_1}^{-1} T_{\gamma_1}
	T_{\delta_1}. \] 
	To prove the lemma, we therefore need to show that $T_\alpha^n\Psi$ is not in the handlebody group for any $n$. To show this claim, we will study the
	effect of $T_\alpha^n\Psi$ on a meridian $\mu$, which is disjoint from $\delta_1$ 
	and intersects $\gamma_2$ in a single point. First
	note that $\mu$ itself, as well as $\delta_1, \gamma_1$ are disjoint from
	$\alpha$, and so the image $T_\alpha^n\Psi(\mu)$ will not depend on $n$. Thus, we may assume
	$n=0$.
	
	The action of the first three twists is shown in
	Figure~\ref{fig:setup-example} on the right. Twists are executed right-to-left, $T_x$ is
	a left-handed twist about the curve $x$.
	
	Instead of actually performing the final twist, we can now
        determine the resulting word in $\pi_1(V)$ by recording
        intersections with a cut system as follows. We choose a cut
        system consisting of three discs $D_1, D_2, D_3$. Here, $D_1$
        is freely homotopic to $\mu$, $D_2$ intersects $\alpha$ in a
        single point, and $D_3$ is disjoint from all curves involved.
        Compare the bottom picture in Figure~\ref{fig:setup-example}.
	We also choose transverse orientations, so that the cut system
        defines an oriented basis $x,y,z$ of $\pi_1(V)$.

	To find the element which
        $T_{\gamma_1}^{-1}T_{\delta_1}^{-1}T_{\gamma_1}T_{\delta_1}\mu$
        defines in $\pi_1(V)$, we now follow along the purple curve,
        starting at the solidly drawn basepoint, turn right and follow
        $\gamma_1$ whenever we encounter $\gamma_1$, and keep track of
        intersections with $D_1, D_2$. With the transverse
        orientations as in Figure~\ref{fig:setup-example}, this yields the
        following word (for readability, intersections due to $\gamma_1$
	are bracketed):

        \[ x(x^{-1}y^{-1})(xy)y(y^{-1}x^{-1})(yx)(y^{-1}x^{-1}) =
        y^{-1}xyx^{-1}yxy^{-1}x^{-1} \] This is a nontrivial element
        in $\pi_1(V)$, and therefore
        $T_{\gamma_1}^{-1}T_{\delta_1}^{-1}T_{\gamma_1}T_{\delta_1}\mu$
        is not a meridian.

\bibliographystyle{math}
\bibliography{stabs}

\begin{thebibliography}{McC2}

\bibitem[Far]{Farb}
Benson Farb.
\newblock {The extrinsic geometry of subgroups and the generalized word
  problem}.
\newblock {\em Proc. London Math. Soc. (3)} {\bf 68}(1994), 577--593.

\bibitem[FM]{Primer}
Benson Farb and Dan Margalit.
\newblock {\em A primer on mapping class groups}, volume~49 of {\em Princeton
  Mathematical Series}.
\newblock Princeton University Press, Princeton, NJ, 2012.

\bibitem[HH1]{HH1}
Ursula Hamenst\"{a}dt and Sebastian Hensel.
\newblock {The geometry of the handlebody groups {I}: distortion}.
\newblock {\em J. Topol. Anal.} {\bf 4}(2012), 71--97.

\bibitem[HH2]{HH2}
Ursula Hamenst\"{a}dt and Sebastian Hensel.
\newblock {The geometry of the handlebody groups {II}: Dehn functions}.
\newblock {\em arXiv:1804.11133} (2018).

\bibitem[HM]{HM}
Michael Handel and Lee Mosher.
\newblock {Lipschitz retraction and distortion for subgroups of {${\rm
  Out}(F_n)$}}.
\newblock {\em Geom. Topol.} {\bf 17}(2013), 1535--1579.

\bibitem[Hem]{Hempel}
John Hempel.
\newblock {3-manifolds as viewed from the curve complex}.
\newblock {\em Topology} {\bf 40}(2001), 631--657.

\bibitem[Hen]{H-Survey}
Sebastian Hensel.
\newblock {A primer on handlebody groups}.
\newblock {\em preprint, available at
  \emph{http://www.mathematik.uni-muenchen.de/~hensel}} (2017).

\bibitem[Lau]{L74}
Fran\c{c}ois Laudenbach.
\newblock {\em Topologie de la dimension trois: homotopie et isotopie}.
\newblock Soci\'{e}t\'{e} Math\'{e}matique de France, Paris, 1974.
\newblock With an English summary and table of contents, Ast\'{e}risque, No.
  12.

\bibitem[Luf]{Luft}
E.~Luft.
\newblock {Actions of the homeotopy group of an orientable {$3$}-dimensional
  handlebody}.
\newblock {\em Math. Ann.} {\bf 234}(1978), 279--292.

\bibitem[MM]{MM2}
H.~A. Masur and Y.~N. Minsky.
\newblock {Geometry of the complex of curves. {II}. {H}ierarchical structure}.
\newblock {\em Geom. Funct. Anal.} {\bf 10}(2000), 902--974.

\bibitem[Mas]{Masur}
Howard Masur.
\newblock {Measured foliations and handlebodies}.
\newblock {\em Ergodic Theory Dynam. Systems} {\bf 6}(1986), 99--116.

\bibitem[McC1]{McCullough}
Darryl McCullough.
\newblock {Twist groups of compact {$3$}-manifolds}.
\newblock {\em Topology} {\bf 24}(1985), 461--474.

\bibitem[McC2]{McC-geom-finite}
Darryl McCullough.
\newblock {Virtually geometrically finite mapping class groups of
  {$3$}-manifolds}.
\newblock {\em J. Differential Geom.} {\bf 33}(1991), 1--65.

\end{thebibliography}

\bigskip
\noindent Sebastian Hensel\\
Mathematisches Institut der Universität München\\
Theresienstraße 39, 80333 München\\
Email: hensel@math.lmu.de

\end{document}